\documentclass[11pt]{amsart}

\usepackage{fullpage}

\usepackage{amsmath, amscd, amsthm, amssymb}

\usepackage{hyperref}

\def\C{{\mathbf C}}
\def\R{{\mathbf R}}
\def\Z{{\mathbf Z}}
\def\Q{{\mathbf Q}}
\def\A{{\mathbf A}}

\newtheorem{theorem}{Theorem}[subsection]

\newtheorem{lemma}[theorem]{Lemma}

\newtheorem{proposition}[theorem]{Proposition}

\newtheorem{corollary}[theorem]{Corollary}

\theoremstyle{definition}
\newtheorem{definition}[theorem]{Definition}

\theoremstyle{remark}
\newtheorem{remark}[theorem]{Remark}

\newcommand{\mm}[4]{\left(\begin{smallmatrix} #1 & #2\\ #3 & #4\end{smallmatrix}\right)}

\DeclareMathOperator{\SO}{SO}

\DeclareMathOperator{\GSpin}{GSpin}
\DeclareMathOperator{\Sp}{Sp}

\DeclareMathOperator{\SU}{SU}

\DeclareMathOperator{\SL}{SL}
\DeclareMathOperator{\GL}{GL}

\DeclareMathOperator{\charf}{char}
\DeclareMathOperator{\diag}{diag}

\DeclareMathOperator{\pr}{pr}

\def\g{{\mathfrak g}}

\def\k{{\mathfrak k}}
\def\p{{\mathfrak p}}

\def\m{{\mathfrak m}}
\def\so{{\mathfrak {so}}}
\def\sl{{\mathfrak {sl}}}

\def\m{{\mathfrak m}}
\def\n{{\mathfrak n}}

\def\Wh{{\mathcal W}}
\def\Vm{{\mathbb{V}}}

\allowdisplaybreaks

\begin{document}
\title{Modular forms on indefinite orthogonal groups of rank three}
\author{Aaron Pollack}
\address{Department of Mathematics\\ Duke University\\ Durham, NC USA}
\email{apollack@math.duke.edu}
\thanks{The author has been supported by the Simons Foundation via Collaboration Grant number 585147.}

\begin{abstract} We develop a theory of modular forms on the groups $\mathrm{SO}(3,n+1)$, $n \geq 3$.  This is very similar to, but simpler, than the notion of modular forms on quaternionic exceptional groups, which was initiated by Gross-Wallach and Gan-Gross-Savin.  We prove the results analogous to those of earlier papers of the author on modular forms on exceptional groups, except now in the familiar setting of classical groups.  Moreover, in the setting of $\mathrm{SO}(3,n+1)$, there is a family of absolutely convergent Eisenstein series, which are modular forms.  We prove that these Eisenstein series have algebraic Fourier coefficients, like the classical holomorphic Eisenstein series on $\mathrm{SO}(2,n)$.  As an application, we prove that the so-called ``next-to-minimal'' modular form on quaternionic $E_8$ has rational Fourier expansion, under a mild local assumption.
\end{abstract}
\maketitle


\section{Introduction} That there is a notion of \emph{modular forms} on the quaternionic exceptional groups goes back to Gross-Wallach and Gan-Gross-Savin.  This theory is based on the so-called quaternionic discrete series, whose study was initiated by Gross-Wallach \cite{grossWallach1}, \cite{grossWallach2}.  The modular forms on the quaternionic exceptional groups have been the subject of the papers \cite{ganGrossSavin, weissman, pollackQDS, pollackE8, pollackG2}. It turns out that there is a completely analogous but much simpler theory of ``modular forms" on the classical groups $\SO(3,n+1)$. (Note that when $n$ is even, these groups do not have discrete series.)  The purpose of this paper is to write down this notion of modular forms, and prove a few of the basic theorems. In particular, we
\begin{enumerate}
\item find the explicit form of the Fourier expansion of such modular forms, in complete analogy with the results of \cite{pollackQDS};
\item prove that certain absolutely convergent degenerate Eisenstein series that are modular forms have algebraic Fourier coefficients.\end{enumerate}

One source of examples of these modular forms comes from certain constant terms of modular forms on the quaternionic exceptional groups.  More precisely, suppose $G_J$ is\footnote{The subscript ``$J$" comes from the fact that these groups are associated to certain cubic Jordan algebras $J$.} a quaternionic exceptional group as in \cite{pollackQDS} with rational root type $F_4$, so that $G_J$ has absolute Dynkin type $F_4, E_6, E_7$ or $E_8$.  Then $G_J$ possesses a maximal parabolic $Q_J = L_J V_J$ with $L_J$ having rational root type $B_3$.  Up to anisotropic factors, $L_J$ is isogenous to a group $\SO(3,n+1)$ where $n = 3,4, 6, 10$ if $G_J$ has type $F_4, E_6, E_7, E_8$, respectively.  One can take the constant term of a modular form of weight $\ell$ on $G_J$ down to $L_J$, and we prove in section \ref{sec:Const} that these constant terms are modular forms of weight $\ell$ on $L_J$.  As an application of the above facts, we prove that the so-called ``next-to-minimal" modular form on quaternionic $E_8$ has rational Fourier expansion.

The result (2) on the Fourier coefficients of Eisenstein series is the analogue of the fact that on $\SO(2,n)$ or another hermitian tube domain, the absolutely convergent holomorphic Eisenstein series have algebraic Fourier expansions.  As this paper shows, the notion of modular forms on $\SO(3,n+1)$ is very similar to that of modular forms on the quaternionic exceptional groups, such as $E_{8,4}$.  However, because $\SO(3,n+1)$ is a classical group, and more importantly because the natural Fourier expansion of modular forms on $\SO(3,n+1)$ takes place along an \emph{abelian} unipotent group, the notion of these modular forms is substantially simpler than that on the exceptional groups.  Thus we hope that $\SO(3,n+1)$ can be used as a test case for developing analogous results on the quaternionic exceptional groups.  In particular, the algebraicity of the Fourier coefficients of the degenerate Eisenstein series from \cite{pollackE8} appears difficult.  Part of the motivation for writing this paper was to get closer to proving that the Fourier coefficients of these Eisenstein series on exceptional groups are algebraic.

The definition of the modular forms on $\SO(3,n+1)$ is very similar to that of the modular forms on the exceptional groups from \cite{pollackQDS}.  In particular, if $V$ is a rational quadratic space of signature $(3,n+1)$, then the maximal compact subgroup $K$ of $\SO(V)(\R)$ is $S(O(3) \times O(n+1))$.  This group maps to $O(3) = (\SU(2)/\mu_2) \rtimes \langle \pm 1 \rangle$.  Denote by $\Vm_{\ell} = Sym^{2\ell}(\C^2)$ the $(2\ell+1)$-dimensional representation of $K$ that factors through $O(3)$.  \emph{Modular forms} on $\SO(V)$ of weight $\ell$ are then $\Vm_{\ell}$-valued automorphic functions $\varphi$ on $\SO(V)(\A)$ that
\begin{enumerate}
\item satisfy $\varphi(gk) = k^{-1} \varphi(g)$ for all $g \in \SO(V)(\A)$ and $k \in K$
\item and are annihilated by a special differential operator $\mathcal{D}_{\ell}$.
\end{enumerate}
The precise definition of modular forms, and in particular of the operator $\mathcal{D}_{\ell}$, is given in section \ref{sec:MFFE} below.  Throughout the paper, $(V,q)$ is rational quadratic space of Witt rank three and signature $(3,n+1)$ over $\R$.  Denote by $(x,y) = q(x+y)-q(x)-q(y)$ denote the associated bilinear form.  We write $V = \Q e \oplus V' \oplus \Q f$ with $V'$ a non-degenerate quadratic space of signature $(2,n)$ and $e, f$ isotropic vectors in $(V')^\perp$ with $(e,f) = 1$.  We let $G=\SO(V)$ act on the left of $V$. 

The first result is the Fourier expansion of modular forms on $G=\SO(V)$, in complete analogy to Theorem 1.2.1 of \cite{pollackQDS}.  Denote by $P = MN$ the parabolic subgroup of $\SO(V)$ that stabilizes the isotropic line $\Q e$, so that $M \simeq \GL_1 \times \SO(V')$ and $N \simeq V'$ is abelian.  Let $n: V' \rightarrow N$ denote this identification, which is specified in section \ref{sec:notation} below.  If $\varphi$ is an automorphic form on $G$, then one has 
\begin{equation}\label{eqn:FE1} \varphi(g) = \sum_{\eta \in V'(\Q)}{\varphi_\eta(g)}\end{equation}
where
\[\varphi_\eta(g) = \int_{V'(\Q)\backslash V'(\A)}{\psi^{-1}((\eta,x))\varphi(n(x)g)\,dx}\]
and $\psi: \Q\backslash \A \rightarrow \C^\times$ is our fixed standard additive character.  

The first result Theorem \ref{thm:IntroFE} is a refinement of the expansion \eqref{eqn:FE1} when $\varphi$ is a modular form of weight $\ell$ on $G$.  See Definition \ref{def:Weta} below for the precise definition of the functions $\Wh_{\eta}: G \rightarrow \Vm_{\ell}$ that appear in Theorem \ref{thm:IntroFE}. They are defined in terms of $K$-Bessel functions, exactly as in the Fourier expansion of the modular forms on the quaternionic exceptional groups in \cite{pollackQDS}. In section \ref{sec:notation} we specify a basis $\{x^{2\ell}, x^{2\ell-1}y, \ldots, xy^{2\ell-1}, y^{2\ell}\}$ of $\Vm_{\ell}$.
 
\begin{theorem}\label{thm:IntroFE} Suppose $\varphi$ is a modular form of weight $\ell \geq 1$ on $G$.  Then for $\eta \in V'(\Q)$ with $q(\eta) \geq 0$, there are locally constant functions $a_{\varphi}(\eta): G(\A_f) \rightarrow \C$ so that
\begin{equation}\label{eqn:ThmFE} 
\varphi(g) = \varphi_0(g) + \sum_{0 \neq \eta \in V'(\Q), q(\eta) \geq 0}{ a_{\varphi}(\eta)(g_f) \Wh_{2\pi \eta}(g_\infty)}
\end{equation}
for every $g = g_f g_\infty$ in $G(\A_f) \times G(\R)$.  Moreover, for $m \in M$, the constant term $\varphi_0$ is of the form
\[\varphi_0(m) = t^{\ell}|t|\left(\Phi(m)x^{2\ell} + \beta(m_f) x^{\ell}y^{\ell} + \Phi'(m) y^{2 \ell}\right)\]
where $\Phi$ is an automorphic function associated to a holomorphic modular form of weight $\ell$ on $M$, $\beta$ is a locally constant function on $M(\A_f)$, and $\Phi'$ is a certain $(K \cap M)$-right translate of $\Phi$.\end{theorem}

The second theorem concerns the Fourier expansion of degenerate Eisenstein series on $G$.  More precisely, if $\ell > n+1$ is even and $n$ is even then there is (a family of) absolutely convergent Eisenstein series $E_{\ell}(g)$, which are modular forms of weight $\ell$ on $G$.  These Eisenstein series are associated to the induction space $Ind_{P}^{G}(\delta_P^{(\ell+1)/(n+2)})$.  See section \ref{sec:Eis} for the precise definition.  These degenerate Eisenstein series are the analogues of the degenerate Heisenberg Eisenstein series considered in \cite{pollackE8} or the classical degenerate holomorphic Siegel Eisenstein series on $\Sp_{2n}$.  Theorem \ref{thm:IntroEis} below states that the Fourier coefficients of $E_{\ell}(g)$ are algebraic numbers.

To set up the result, suppose that $\ell > 0$ is even, and $\varphi$ a modular form on $G$ of weight $\ell$.  Let the Fourier expansion of $\varphi$ be as in \eqref{eqn:ThmFE}. We say that $\varphi$ has Fourier coefficients in a field $E$ if
\begin{enumerate}
\item The locally constant functions $a_{\varphi}(\eta)$ are valued in $E$;
\item The holomorphic modular form associated to $\Phi$ has Fourier coefficients in $E$;
\item The locally constant function $\beta$ is valued in $E \cdot \frac{\zeta(\ell+1)}{(2\pi)^{\ell}}$.
\end{enumerate}
The perhaps unusual-looking normalization of the constant $\beta$ is dictated, similar to the results of \cite{pollackE8}, by the fact that the modular forms one constructs in practice have Fourier coefficients valued in some fixed field $E$ as in the above definition.

\begin{theorem}\label{thm:IntroEis} Suppose that $\dim(V') = n+2$ is a multiple of $4$ and that $\ell > n+1$ is even.  Then the Eisenstein series $E_{\ell}(g)$ has $\overline{\Q}$-valued Fourier coefficients.\end{theorem}

Note that under the assumptions of Theorem \ref{thm:IntroEis}, the group $G(\R) = \SO(3,n+1)$ does not possess discrete series.  Nevertheless, the modular forms exist and one can prove that the most basic modular forms--the degenerate Eisenstein series--have algebraic Fourier coefficients.  See Theorem \ref{thm:EisAlgFC} for the precise statement.

The main application of the above results is to the next-to-minimal modular form on quaternionic $E_8$.  Namely, we prove that it has rational Fourier expansion, under a certain mild local assumption. The result is as follows.  In the statement of the theorem, the group $G_J$ is the $\Q$-group of type $E_8$ from, e.g., \cite{pollackQDS} or \cite{pollackE8}, that has rational root system of type $F_4$.

\begin{theorem}\label{thm:introNTM} Let $E_J(g,s;8)$ denote the degenerate Heisenberg Eisenstein series on $G_J$ that is spherical at every finite place and ``weight $8$" at infinity.  Then $E_J(g,s;8)$ is regular at $s=9$ and defines a square integrable modular form of weight $8$ at this point.  The modular form $\theta_{ntm}(g)=E_J(g,s=9;8)$ has rational rank zero, rank one, and rank two Fourier coefficients.  In particular, if Property $V$ of section \ref{sec:ntm} holds for a single finite prime $p$, the modular form $E_J(g,s=9;8)$ has rational Fourier expansion.\end{theorem}

See Theorem \ref{thm:ntmRat} below for the precise statement.  The Eisenstein series $E_J(g,s;\ell)$ are the subject of \cite{pollackE8}.  The modular form $E_J(g,s=9;8)$ is expected to be the next-to-minimal modular form on $G_J$.  In the case of split $E_8$, the next-to-minimal automorphic representation has been considered recently in \cite{GGKPS}.  

The \emph{minimal} modular form $\theta_{Gan}$ on quaternionic $E_8$ was considered in \cite{ganATM, ganSW, pollackE8}; it is of weight $4$.  It is the $E_8$-analogue of Kim's weight $4$ exceptional modular form \cite{kim} on $GE_{7,3}$, and in fact Kim's weight $4$ modular form appears in the constant term of $\theta_{Gan}$ along the unipotent radical of the Heisenberg parabolic.  The \emph{next-to-minimal} modular form $\theta_{ntm}$ that is the subject of Theorem \ref{thm:introNTM} is weight $8$ and is the analogue of Kim's weight $8$ singular modular form on $GE_{7,3}$ from \cite{kim}.  Moreover, Kim's singular modular form shows up in the constant term of $\theta_{ntm}$ along the same Heisenberg parabolic.

The final result we give is to the minimal modular form on the groups $\SO(3,8k+3)$ and to a so-called distinguished modular form on $\SO(3,8k+2)$.  This is done is section \ref{sec:minD}, and is the analogue of the results in \cite{pollackE8} to the classical groups of type $D_{4k+3}$.  Specifically, we prove the following theorem; see Theorem \ref{thm:minTypeD} below.

\begin{theorem} Let $k \geq 1$ be an integer, and let $G$ be the $\Q$-group of type $D_{4k+3}$ that is split at every finite place and $\SO(3,8k+3)$ at infinity.  The Eisenstein series $E_{4k}(g,s)$ is regular at $s=4k+1$.  The value $\theta(g) = E_{4k}(g,s=4k+1)$ is a modular form on $G$ of weight $4k$, having rational Fourier expansion with all non-degenerate Fourier coefficients equal to $0$.  Its restriction $\theta'$ to groups $G' = \SO(3,8k+2) \subseteq G$ is a modular form of weight $4k$ that is distinguished.\end{theorem}

\subsection{Acknowledgements} It is a pleasure to thank Solomon Friedberg, Dmitry Gourevitch, Henrik Gustafsson, Axel Kleinschmidt, Daniel Persson, Boris Pioline, Siddhartha Sahi and Gordan Savin for helpful conversations and correspondence related to this work.  It is also a pleasure to thank the Simons Center for Geometry and Physics, and the workshop Automorphic Structures in String Theory (2019) for hosting an inspiring conference where many of these conversations took place.  

\section{Notation}\label{sec:notation} In this section we define the notation that we will use throughout the paper. Let $(V_2,q_2)$ denote a two-dimensional rational quadratic space with positive definite quadratic form.  Similarly, let $(V_n,q_n)$ denote an $n$-dimensional rational quadratic space with positive definite quadratic form.  Set $V' = V_2 \oplus V_n$ with quadratic form $q(x,y) = q_2(x) - q_n(y)$, so that $V'$ has signature $(2,n)$.  We set $V = \Q e \oplus V' \oplus \Q f$ with quadratic form $q(\alpha e + v' + \beta f) = \alpha \beta + q'(v')$.  Thus $V$ has signature $(3,n+1)$.  For some of the results below, we will assume $V'$ has Witt rank two, although this is not necessary everywhere.

Let $\iota$ be the involution on $V$ given by $\iota(\alpha e + x + y + \beta f) = \beta e + x - y +\alpha f$, where $x \in V_2$, $y \in V_n$.  Then $(v,\iota(v)) \geq 0$, and conjugation by $\iota$ is a Cartan involution $\theta_\iota$ on $\SO(V)(\R)$.  We set $u_{+} = e+f$ and $u_{-} = e-f$ so that $q(u_+) = 1 $ and $q(u_-) = -1$.  We let $v_1, v_2$ be an orthonormal basis of $V_2(\R)$ so that $(v_i,v_j) = \delta_{ij}$, and $\{u_1, u_2, \ldots, u_n\}$ be a basis of $V_n$.

We set $V_3 = V_2 \oplus \R u_+$ and $V_{n+1} = V_n \oplus \R e_{-}$.  The induced Cartan involution on the Lie algebra $\g_0 = \so(V)$ produces the decomposition $\g_0 = \k_0 \oplus \p_0$ with $\k_0 = \g_0^{\theta_\iota=1}$ and $\p_0 = \g_0^{\theta_\iota = -1}$.  Under the isomorphism $\g_0 \simeq \wedge^2 V$, one has $\p_0 = V_3 \otimes V_{n+1} \subseteq \wedge^2 V$ and $\k_0 = \wedge^2 V_3 \oplus \wedge^2 V_{n+1} \subseteq \wedge^2 V$.  We set $\p = \p_0 \otimes \C$ and $\k = \k_0 \otimes \C$.

The Lie algebra $\wedge^2 V_3 \otimes \C \subseteq \k$ is isomorphic to $\sl_2(\C)$.  For an $\sl_2$-triple $(E,H,F)$ in $\wedge^2 V \otimes \C$, one can take $E = (iv_1-v_2) \wedge u_{+}/\sqrt{2}$, $H = -2i v_1 \wedge v_2$, $F = (i v_1 + v_2) \wedge u_{+}/\sqrt{2}$.  Then $[E,F] = H$, $[H,E] = 2E$ and $[H,F] = -2F$, so that indeed $(E, H, F)$ is an $\sl_2$-triple.

Denote by $P = MN$ the parabolic subgroup of $G$ that fixes the line $\Q e$.  We are letting $G$ act on the left of $V$.  Denote by $\nu: P \rightarrow \GL_1$ the character so that $p e = \nu(p) e$. We let $M$ be the Levi subgroup that also fixes the line $\Q f$.  Denote by $N$ the unipotent radical of $P$.  Then $N \simeq V'$ is abelian, and for $x \in V'$, we set $n(x) = \exp( e \wedge x)$.  Thus $n: V' \rightarrow N$ is an isomorphism.  One has $n(x) = \exp(e \wedge x)$ takes $e \mapsto e$, $v \mapsto v + (x,v) e$ if $v \in V'$, and $f \mapsto f - x - \frac{1}{2}(x,x) e$.  The matrix corresponding to $n(x)$ is $\left(\begin{array}{ccc} 1 & \,^tx & -(x,x)/2 \\ & 1 & -x \\ && 1 \end{array}\right)$.  

As mentioned in the introduction, we let $\Vm_\ell$ denote the $(2\ell+1)$-dimensional representation of $K \subseteq \SO(3,n+1)$ that factors through $O(3)$.  Let $x,y$ be a fixed weight basis of the two-dimensional representation $Y_2 \simeq \C^2$ of $\wedge^2 V_3 \otimes \C \simeq \sl_2(\C)$.  We may identify $V_3 \otimes \C$ with the symmetric square representation $S^2(Y_2)$ of this two-dimensional representation, which has basis $\{x^2, xy, y^2\}$.  We choose this weight basis $x,y$ and the identification $S^2(Y_2) \simeq V_3 \otimes \C$ so that $x^2$ corresponds to $iv_1 -v_2$, $xy$ corresponds to $u_+/\sqrt{2}$, and $y^2$ corresponds to $iv_1 + v_2$.

Throughout the paper, the letter $H$ denotes a hyperbolic plane.  Moreover, we frequently use the subscript $0$ to denote an integral lattice inside a rational quadratic space.  Thus, for example $H_0 \cong \Z \oplus \Z$.

\section{Modular forms and their Fourier expansion}\label{sec:MFFE} In this section we define the modular forms on $G = \SO(V)$, and give the explicit form of their Fourier expansion.  That is, we prove Theorem \ref{thm:IntroFE} of the introduction.

\subsection{Definition of modular forms}\label{subsec:MFdef} We now define modular forms on $G = \SO(V)$.  As mentioned in the introduction, a modular form on $G$ of weight $\ell$ is an automorphic function $\varphi: G(\Q)\backslash G(\A) \rightarrow \Vm_\ell$ of moderate growth satisfying
\begin{enumerate}
\item $\varphi(gk) = k^{-1}\cdot \varphi(g)$ for all $g \in G(\A)$ and $k \in K$
\item $D_{\ell} \varphi \equiv 0$ for a certain linear differential operator $D_{\ell}$ defined below.
\end{enumerate}

To define the differential operator $D_{\ell}$, let $X_\gamma$ be a basis of $\p$ and $X_\gamma^\vee$ be the dual basis of $\p^\vee$.  Suppose $\varphi: G(\A) \rightarrow \Vm_\ell$ satisfies $\varphi(gk) = k^{-1} \varphi(g)$.  Define $\widetilde{D_{\ell}}\varphi = \sum_{\gamma}{X_\gamma \phi \otimes X_\gamma^\vee}$, which is valued in $\Vm_{\ell} \otimes \p^\vee$.  Here $X_\gamma \varphi$ denotes the right-regular action of $\p$ on $\varphi$.  Note that
\[\Vm_\ell \otimes \p^\vee = (S^{2\ell}(Y_2) \otimes S^2(Y_2)) \boxtimes V_{n+1} = (S^{2\ell+2}(Y_2) \oplus S^{2\ell}(Y_2) \oplus S^{2\ell-2}(Y_2)) \boxtimes V_{n+1}.\]
Denote by $\pr$ the $K$-equivariant projection $\Vm_\ell \otimes \p^\vee \rightarrow (S^{2\ell}(Y_2) \oplus S^{2\ell-2}(Y_2)) \boxtimes V_{n+1}$.  We define $D_{\ell} = \pr \circ \widetilde{D}$. 

Note that $S^2(Y_2) \subseteq S^1(Y_2) \otimes S^1(Y_2)$, and thus $\pr$ is also the composition
\begin{align*} \Vm_\ell \otimes \p^\vee  \subseteq (S^{2\ell}(Y_2) \otimes S^1(Y_2) \otimes S^1(Y_2)) \boxtimes V_{n+1} &= (S^{2\ell+1}(Y_2) \oplus S^{2\ell-1}(Y_2)) \otimes S^1(Y_2) \boxtimes V_{n+1}\\ &\rightarrow S^{2\ell-1}(Y_2) \otimes (S^1(Y_2) \boxtimes V_{n+1}).\end{align*}
This last line makes clear the analogy between modular forms on $\SO(3,n+1)$ and modular forms in the sense of \cite{pollackQDS}.

\subsection{The Fourier expansion of modular forms}\label{subsec:FEMF} In this subsection we give the precise Fourier expansion of modular forms on $G$.  More precisely, suppose $\ell \geq 1$, $\eta \in V'(\R)$.  We say that a function $F: G(\R) \rightarrow \Vm_{\ell}$ is a generalized Whittaker function of type $\eta$ if $F$ is of moderate growth and satisfies
\begin{enumerate}
\item $F(n(x)g) = e^{i (\eta,x)} F(g)$ 
\item $F(gk) = k^{-1} \cdot F(g)$
\item $D_{\ell} F(g) = 0$
\end{enumerate}
for all $g \in G(\R)$, $k \in K$ and $x \in V'(\R)$.  In this subsection, we completely characterize the generalized Whittaker functions of type $\eta$, for all $\eta \in V'(\R)$.  In particular, we prove that if $q(\eta) < 0$, the only such function is the $0$ function, while if $\eta \neq 0$ and $q(\eta) \geq 0$ then all such functions are scalar multiples of the function $\Wh_{\eta}$ mentioned in the introduction.

In order to understand these generalized Whittaker functions, we make relatively explicit the differential equation $D_{\ell} F = 0$ in coordinates.  To do this, we begin by making an explicit Iwasawa decomposition of some elements of the Lie algebra of $G$.  In more detail, let $\n, \m$ denote the complexified Lie algebras of $N$, $M$; one has a decomposition $\g = \n + \m + \k$.  We have
\[\p = (\R u_{+} \oplus V_2) \wedge (\R u_{-} \oplus V_n) = \R u_+ \wedge u_{-} \oplus u_{+} \wedge V_n \oplus V_2 \wedge u_{-} \oplus V_2 \wedge V_n.\]
In $\n +\m + \k$ coordinates, a basis of $\p$ decomposes as follows:
\begin{itemize}
\item $u_{+} \wedge u_{-} = (e + f) \wedge (e-f) = -2 e \wedge f \in \m$.
\item $u_{+} \wedge u_j = (e+f) \wedge u_j = (2e - u_{-}) \wedge u_j = 2 e \wedge u_j - u_{-} \wedge u_j \in \n + \k$. (Recall that the $u_j$ are a basis of $V_n$.)
\item $v_i \wedge u_j \in \m$. (Recall that $v_1, v_2$ is a basis of $V_2$.)
\item $v_i \wedge u_{-} = v_i \wedge (e-f) = v_i \wedge (2 e- u_{+}) = -2 e \wedge v_i + u_{+} \wedge v_i \in \n + \k$.
\end{itemize}

For ease of notation, let $[x^j] = \frac{x^j}{j!}$ and similarly $[y^j] = \frac{y^j}{j!}$.  Let $F_{v}$ denote the components of the $\Vm_{\ell}$-valued function $F$; that is 
\[F = \sum_{-\ell \leq v \leq \ell}{F_v [x^{\ell +v}][y^{\ell-v}]}.\]
Let $\{u_1^\vee, \ldots, u_n^\vee\}$ be the basis dual to the basis $\{u_1, \ldots, u_n\}$ and $u_{-}^\vee$ dual to $u_{-}$.  Denote by $D^{M}_{iv_1 \pm v_2, u_j}$ the differential operator on functions on $M$ corresponding to the (differential right-regular) action of $(iv_1 \pm v_2) \wedge u_j$ on $F$.  For future reference, note that $(iv_1-v_2, iv_1+v_2) = -2$. 

Suppose $t \in \R^\times$, $m \in \SO(2,n)$ and $x \in V'(\R)$ so that $n(x)\diag(t,m,t^{-1}) \in N(\R)M(\R) = P(\R)$.  Restricting the function $F$ to $P$, we write $F(x,t,m) := F(n(x)\diag(t,m,t^{-1}))$.  For $w \in V'$, denote 
\[D^{V'}_{w}F(x,t,m) = \frac{d}{d\lambda} F(x + \lambda w, t,m)|_{\lambda = 0}\]
the partial derivative in the $w$-direction.  Also, note that $(e \wedge f) F = t\partial_t F$. 

Suppose $F:G(\R) \rightarrow \Vm_{\ell}$ is a function satisfying $F(gk) = k^{-1} F(g)$ for all $g \in G(\R)$ and $k \in K$.  The following proposition computes $D_{\ell}F(x,t,m)$ explicitly in coordinates, in terms of the differential operators $D^{M}$, $D^{V'}$ and $t\partial_{t}$.  To state the result, note that the operator $D_{\ell}$ is valued in $S^{2\ell-1}(Y_2) \otimes (Y_2 \boxtimes V_{n+1})$, which has a basis consisting of elements $[x^{\ell+v-1}][y^{\ell-v}] \otimes y \otimes u_{-}^\vee$, $[x^{\ell+v-1}][y^{\ell-v}] \otimes x \otimes u_{-}^\vee$, $[x^{\ell+v-1}][y^{\ell-v}] \otimes y \otimes u_{j}^\vee$, $[x^{\ell+v-1}][y^{\ell-v}] \otimes x \otimes u_{j}^\vee$.  

\begin{proposition}\label{thm:Dcoefs} Suppose $F:G(\R) \rightarrow \Vm_{\ell}$ is a function satisfying $F(gk) = k^{-1} F(g)$ for all $g \in G(\R)$ and $k \in K$. The coefficients of linear independent terms in $2 D_{\ell}F$ are as follows:
\begin{enumerate}
\item $[x^{\ell+v-1}][y^{\ell-v}] \otimes y \otimes u_{-}^\vee:$
\[2 D_{tm (iv_1-v_2)}^{V'}F_v - \sqrt{2}(\ell+v) F_{v-1} + \sqrt{2} t\partial_t F_{v-1}\]
\item $[x^{\ell+v-1}][y^{\ell-v}] \otimes x \otimes u_{-}^\vee:$
\[-\sqrt{2}t \partial_t F_v - 2D^{V'}_{tm (iv_1+v_2)}F_{v-1} + \sqrt{2}(\ell-v+1)F_v\]
\item $[x^{\ell+v-1}][y^{\ell-v}] \otimes y \otimes u_{j}^\vee:$
\[-D^M_{iv_1-v_2,u_j} F_v - \sqrt{2} D^{V'}_{tm u_j}F_{v-1}\]
\item $[x^{\ell+v-1}][y^{\ell-v}] \otimes x \otimes u_{j}^\vee:$
\[\sqrt{2} D^{V'}_{tm u_j} F_v + D^M_{iv_1 + v_2,u_j} F_{v-1}.\]
\end{enumerate}
\end{proposition}
\begin{proof} As the computation is straightforward, we relegate the details to the appendix.  See subsection \ref{subsec:appFE}.\end{proof}

As a corollary of the above proposition, we obtain the complete description of the generalized Whittaker functions of type $\eta$.  Thus suppose $\eta \in V'(\R)$ and $F$ is a generalized Whittaker function of type $\eta$.  That is, assume $F$ is of moderate growth and $F(x,t,m)$ satisfies $F(x+w,t,m) = e^{i(\eta,w)}F(x,t,m)$ for all $w \in V'$, so that $D^{V'}_{w}F = i(\eta,w)F$.

To state the theorem, we first define the function $\Wh_{\eta}$ that plays a crucial role in this paper. 
\begin{definition}\label{def:Weta} Suppose $\eta \in V'(\R)$, $\eta \neq 0$, and $(\eta,\eta) \geq 0$.  For $t \in \GL_1(\R)$, $m \in \SO(V')(\R)$ set
\[ u_\eta(t,m) = \sqrt{2} ti (\eta,m(iv_1-v_2)).\]
Define
\[\Wh_{\eta}(t,m) = t^{\ell} |t| \sum_{-\ell \leq v \leq \ell}{\left(\frac{|u_\eta(t,m)|}{u_\eta(t,m)}\right)^{v} K_v(|u_\eta(t,m)|)}.\]
\end{definition}

Here recall the $K$-Bessel function $K_v(y)$ is defined as
\[K_v(y) = \frac{1}{2} \int_{0}^{\infty}{e^{-y(t+t^{-1})/2} t^{v}\,\frac{dt}{t}}.\]
It satisfies the differential equation $(y\partial_y)^2 K_v(y) = (v^2+y^2)K_v(y)$, diverges at $y \rightarrow 0$ and is of rapid decay as $y \rightarrow \infty$. As $K_v(y)$ diverges for $y \rightarrow 0$, Definition \ref{def:Weta} only makes sense because of the following lemma.

\begin{lemma}\label{lem:etaPos} Suppose $\eta \in V'(\R)$ is such that $(\eta, m(iv_1-v_2)) \neq 0$ for all $m \in \SO(V')(\R)$.  Then $(\eta,\eta) \geq 0$.  Conversely, if $\eta \neq 0$ and $(\eta, \eta) \geq 0$ then $(\eta, m(iv_1-v_2)) \neq 0$ for every $m \in \SO(V')(\R)$.\end{lemma}
\begin{proof} The hypothesis $(\eta, m(iv_1-v_2)) \neq 0$ for every $m \in \SO(V')(\R)$ is equivalent to the statement that the projection of $\eta$ to every positive definite $2$-subspace of $V'$ is nonzero.  Suppose first that $(\eta,\eta) \geq 0$.  Set $\eta' = m^{-1} \eta$.  Then $(\eta',\eta') \geq 0$.  Thus the projection of $\eta'$ to $V_2 = \mathrm{Span}\{v_1,v_2\}$ is not $0$, because otherwise $\eta'$ would lie in $V_n = (V_2)^\perp$ which would imply $(\eta',\eta') < 0$.

Conversely, suppose that $(\eta, \eta) < 0$.  Then $(\R \eta)^\perp$ contains a positive definite $2$-plane $V_2' = m V_2$ for some $m \in \SO(V')(\R)$.  Then $(\eta, m(iv_1-v_2)) = 0$, as desired.\end{proof}

With the above notation, we have the following result.
\begin{theorem}\label{thm:KBessel} Suppose that $F$ is a generalized Whittaker function of type $\eta$ as above. Assume $\eta \neq 0$. If $(\eta, \eta) < 0$, then $F$ is identically $0$.  Conversely, if $(\eta, \eta) \geq 0$, then $F(t,m) = C \Wh_{\eta}(t,m)$ for a constant $C \in \C$.\end{theorem}
\begin{proof} We explain here that the function $\Wh_{\eta}$ has the correct $(K \cap M)$-equivariance property.  See section \ref{subsec:appFE} for the rest of the proof.

We have $K \cap M = \mu_2 \times S(O(2) \times O(n))$.  Consider the element $\epsilon=\diag(-1,1,-1)$ in $M \cap K$.  Then $\epsilon u_+ = -u_+$ while $\epsilon$ acts as the identity on $V_2 = \R v_1 \oplus \R v_2$.  Thus $\epsilon$ acts on $V_3 \simeq S^2(Y_2)$ as $\epsilon x^2 = x^2$, $\epsilon xy = - xy$ and $\epsilon y^2 = y^2$.  It follows that on $\Vm_\ell = S^{2\ell}(Y_2)$ one has $\epsilon x^{\ell+v}y^{\ell-v} = (-1)^{\ell+v} x^{\ell+v}y^{\ell-v}$.  Thus $F_v(x,-t,m) = (-1)^{\ell+v}F_v(x,t,m)$, from which the formula $t^{\ell}|t|$ follows.

Let us consider the equivariance for the $\SO(2)$ part.  Normalize the isomorphism $z: \SO(V_2) \simeq S^1$ by $k(v_1+iv_2) = z(k) (v_1 + iv_2)$.  Then $k(v_1 -iv_2) = z(k)^{-1}(v_1 -iv_2)$ and we have $k (x^{n+v}y^{n-v}) = z(k)^v x^{n+v}y^{n-v}$.  As $F_v(t,m k) = z(k)^{-v} F_v(t,m)$, the $(K \cap M)$-equivariance follows for $k \in \SO(2) \times SO(n)$.  For the nontrivial element of $\pi_0(S(O(2) \times O(n)))$, set $\epsilon'$ to be any element of $S(O(2) \times O(n))$ with $\epsilon' v_1 = v_1$ and $\epsilon' v_2 = -v_2$.  Then, on the one hand, $\epsilon'(x^2) = y^2$,  $\epsilon'(y^2) = x^2$ and $\epsilon'(xy) = xy$, from which it follows that $\epsilon'(x^{n+v}y^{n-v}) = x^{n-v}y^{n+v}$.  On the other hand, 
\[F_v(t,m\epsilon') = t^{\ell}|t| \left(\frac{|u_\eta(t,m)|}{u_\eta(t,m)^*}\right)^{v} K_v(|u_\eta(t,m)|) = F_{-v}(t,m),\]
from which the $(K \cap M)$-equivariance follows for this element. \end{proof}

We now spell out what the generalized Whittaker functions of type $\eta$ look like when $\eta = 0$.  For $k \in \SO(2) \times \SO(n)$, recall that $z(k) \in S^1 \subseteq \C^\times$ is defined by the equality $k (v_1 + iv_2) = z(k) (v_1+iv_2)$.  Additionally, denote by $\epsilon'$ an element of $S(O(2) \times O(n)) \subseteq K \cap M$ with $\epsilon' v_1 = v_1$ and $\epsilon' v_2 = -v_2$.
\begin{corollary}\label{cor:eta0} Suppose $F$ is a generalized Whittaker function of type $\eta = 0$.  Then $F_v(t,m) = 0$ if $v \notin \{-\ell, 0, \ell\}$.  On $M(\R)$, one has $F_0(t,m) = \beta t^{\ell}|t|$, $F_{\pm \ell}(t,m) = |t| F_{\pm \ell}'(m)$ for some constant $\beta \in \C$ and functions $F_{\pm \ell}'(m)$ that are independent of $t$.  The functions $F'_{\pm \ell}(m)$ satisfy $D_{iv_1-v_2,u}^M F_{\ell}'(m) = 0$ and $D^M_{iv_1+v_2,u}F_{-\ell}'(m) = 0$ for all $u \in V_n$.  Moreover, $F_{\pm \ell}'(mk) = z(k)^{\mp \ell}F'_{\pm \ell}(m)$ for all $k \in \SO(2) \times \SO(n)$ and $F_{-\ell}(t,m) = F_{\ell}(t,m\epsilon')$. 

Conversely, suppose $F_{\ell}'(m)$ satisfies $F_{\ell}'(mk) = z(k)^{-\ell}F_{\ell}'(m)$ for all $k \in \SO(2) \times \SO(n)$ and $D^M_{iv_1+v_2,u} F'_{\ell}(m) = 0$ for all $u \in V_n$.  Define $F_{\ell}(t,m)= |t|F'_{\ell}(m)$, $F_{-\ell}(t,m) = F(t,m\epsilon')$, and $F_0(t,m) = \beta t^{\ell} |t|$ for any constant $\beta \in \C$.  Then $F(t,m) = \sum_{-\ell \leq v \leq \ell}{F_v(t,m) [x^{\ell+v}][y^{\ell-v}]}$ is $(K \cap M)$-equivariant and satisfies the differential equations of Proposition \ref{thm:Dcoefs}.\end{corollary}

\begin{proof} First suppose that $(t,m) \in M(\R)^0$, the connected component of the identity.  If $-\ell+1 \leq v \leq \ell-1$, then we have $(t \partial_t - (\ell+v+1))F_{v} = 0$ and $(-t\partial_t + (\ell-v+1))F_v = 0$. Adding the equations gives $-2v F_v = 0$, so $F_v = 0$ unless $v= -\ell, 0,$ or $\ell$.  Because $\eta = 0$, we obtain $D^{M}_{iv_1 \pm v_2,u_j} F_0 = 0$. As in the proof of Theorem \ref{thm:KBessel}, the $(K \cap M)$-equivariance now implies $F_0(t,m) = t^{\ell+1}$ on $M(\R)^0$.  The formulas for $F_{\pm \ell}(t,m)$ on $M(\R)^0$ follow easily.  Additionally, the absolute values $|t|$ and the relationship between $F_{\ell}(t,m)$ and $F_{-\ell}(t,m)$ follow from the $(K \cap M)$-equivariance as in the proof of Theorem \ref{thm:KBessel}.

The converse follows easily, using the formulas for the $(K \cap M)$-action on $\Vm_{\ell}$ from the proof of Theorem \ref{thm:KBessel}. \end{proof}

Below, we will require the following lemma. Denote by $f_{\ell}^1(g,s)$ the $\Vm_{\ell}$-valued, $K$-equivariant inducing section in $Ind_{P(\R)}^{G(\R)}(|\nu|^{s})$, whose restriction to $M(\R)$ is $f_{\ell}^1((t,m,t^{-1}),s) = |t|^s [x^{\ell}][y^{\ell}]$.
\begin{lemma}\label{lem:Dfl} Denote by $f_{\ell}^2(g,s)$ the $\Vm_{\ell}$-valued, $K$-equivariant inducing section in $Ind_{P(\R)}^{G(\R)}(|\nu|^{s})$, whose restriction to $M(\R)$ is $|t|^{s} \left([x^{\ell}][y^{\ell-1}]\otimes y - [x^{\ell-1}][y^{\ell}] \otimes x\right) \otimes u_{-}^{\vee}$. Then 
\[\sqrt{2} D_{\ell} f^{1}_{\ell}(g,s) = (s-\ell-1) f_{\ell}^2(g,s).\]
\end{lemma}
\begin{proof} From Proposition \ref{thm:Dcoefs}, on $M(\R)^0$ one has
\begin{align*} \sqrt{2} D_{\ell} f^{1}_{\ell}((t,m,t^{-1}),s) &= (t\partial_t - (\ell+1))(t^s)\left([x^{\ell}][y^{\ell-1}]\otimes y - [x^{\ell-1}][y^{\ell}] \otimes x\right) \otimes u_{-}^{\vee} \\ &= (s-\ell-1) t^{s} \left([x^{\ell}][y^{\ell-1}]\otimes y - [x^{\ell-1}][y^{\ell}] \otimes x\right) \otimes u_{-}^{\vee}\end{align*}
using that $D^M_{iv_1 \pm v_2, u_j} f_{\ell}^{1}((t,m,t^{-1}),s) = 0$ because $f_{\ell}^{1}$ is independent of the variable $m \in \SO(V')(\R)$.  The lemma follows from the $(K\cap M)$-equivariance.
\end{proof}

\section{The Fourier expansion of Eisenstein series}\label{sec:Eis} There is a $\Vm_\ell$-valued degenerate Eisenstein series on $G$, $E_\ell(g,s)$ associated to the (non-normalized) induction $Ind_{P}^G(|\nu|^{s})$.  If $\ell$ is even, then at $s=\ell+1$ and for appropriate inducing data, this Eisenstein series is a modular form in sense of subsection \ref{subsec:MFdef}.  The purpose of this section is to prove that indeed we get a modular form as above, and to compute the Fourier expansion of this Eisenstein series $E_\ell(g,s=\ell+1)$ along the unipotent radical $N$, at least when $\dim(V')$ is a multiple of four and the Eisenstein series is absolutely convergent.

The Eisenstein series $E_{\ell}(g,s)$ is defined using the inducing section $f_{\ell;\infty}(g,s) := f^1_{\ell}(g,s)$ of Lemma \ref{lem:Dfl} at the archimedean place.  The computation of its Fourier expansion consists of various parts, which we break into subsections.  Let us describe these parts now, before getting into the computation. 

To define some terminology, note that the non-constant Fourier coefficients of a modular form $\varphi$ of weight $\ell$ are parametrized by $\eta \in V'(\R)$, which can be either isotropic or anisotropic.  We call the Fourier coefficients corresponding to the nonzero isotropic $\eta$ \emph{rank one} Fourier coefficients, while those corresponding to the anisotropic $\eta$ the \emph{rank two} Fourier coefficients. 
\begin{enumerate}
\item By applying Lemma \ref{lem:Dfl}, it is immediate to see that if the Eisenstein series is absolutely convergent at $s = \ell+1$ (which occurs if $\ell +1 > \dim(V') = n+2$), then $E_{\ell}(g,s=\ell+1)$ is a modular form of weight $\ell$ for $G$.
\item If the Eisenstein series is not absolutely convergent, then it is not clear--and not necessarily true--that $E_{\ell}(g,s)$ is a modular form at $s = \ell+1$.  To see when it is, we make various archimedean intertwiner computations in subsection \ref{subsec:inter}. Although this is not needed for the Fourier expansion of the absolutely convergent Eisenstein series, it is useful for other applications.
\item We then compute the constant term of the absolutely convergent Eisenstein series $E_{\ell}(g,s=\ell+1)$ in subsection \ref{subsec:CT}.  Similar to what occurs with the degenerate Heisenberg Eisenstein series considered in \cite{pollackE8}, this constant term is a sum of a holomorphic weight $\ell$ degenerate Eisenstein series on $\SO(V')$ and a constant function.
\item The rank one Fourier coefficients of the Eisenstein series $E_{\ell}(g,s=\ell+1)$ are computed exactly as are the rank one Fourier coefficents of the degenerate Heisenberg Eisenstein series of \cite{pollackE8}.  We state the results in subsection \ref{subsec:rank1FC}.
\item The computation of the rank two Fourier coefficients of $E_{\ell}(g,s=\ell+1)$ splits into two parts, a finite part and an archimedean part. The finite part can be extracted from the literature (e.g. \cite{shurman}).  For the convenience of the reader, we give this computation in subsection \ref{subsec:ftePart}.
\item The archimedean part of the computation of the rank two Fourier coefficients of the Eisenstein series $E_{\ell}(g,s=\ell+1)$ is the main theorem of the paper.  This computation is done in subsection \ref{subsec:arch}.  Denote by $w$ the Weyl group element of $G$ that exchanges the parabolic $P$ with its opposite.  Then one has a function on $V'(\R)$ given by
\[x \mapsto f_{\ell}(wn(x);s=\ell+1).\]
What is computed in subsection \ref{subsec:arch} is the Fourier transform of this function.
\end{enumerate}

We now define the Eisenstein series $E_{\ell}(g,\Phi_f,s)$ that is the object of what follows.  Specifically, suppose $\Phi_f$ is a Schwartz-Bruhat function on $V(\A_f)$.  For $g_f \in \SO(V)(\A_f)$, define
\[f_{fte}(g_f,\Phi_f,s) = \int_{\GL_1(\A_f)}{|t|^{s}\Phi_f(t g_f^{-1} e)\,dt}.\]
Now for $g = g_f g_\infty \in G(\A_f) \times G(\R)$, let $f_{\ell}(g,\Phi_f,s) = f_{fte}(g_f,\Phi_f,s) f_{\ell;\infty}(g_\infty,s)$ and set $E_{\ell}(g,\Phi_f,s) = \sum_{\gamma \in P(\Q)\backslash G(\Q)}{f(\gamma g,\Phi_f,s)}$ the Eisenstein series.  When the Schwartz-Bruhat function $\Phi_f$ is $\Q$-valued (or $\overline{\Q}$-valued), these are the Eisenstein series that are the subject of Theorem \ref{thm:IntroFE} and we will prove that the Fourier coefficients of $(2\pi)^{-\ell} E_{\ell}(g,\Phi_f,s=\ell+1)$ are $\overline{\Q}$-valued.  

\subsection{Archimedean intertwiners}\label{subsec:inter} In this subsection we compute some archimedean intertwining operators.  Specifically we compute the intertwining operator
\[M_\infty(w,s) f_{\ell,\infty}(g,s) = \int_{V'(\R)}{f_{\ell,\infty}(wn(x)g,s)\,dx}.\]
This is the content of Proposition \ref{prop:longInter} below.

We begin with the following well-known lemma, which computes a spherical Archimedean intertwiner on the groups $\SO(N,1)$.
\begin{lemma}\label{lem:SO(N,1)} Suppose $U$ is a positive definite quadratic space, and $V_1 = H \oplus U = \R e_1 \oplus U \oplus \R f_1$ is the orthogonal direct sum of $U$ and a hyperbolic plane $H = \R e_1 \oplus \R f_1$.  Denote by $\iota_1$ the involution on $V_1$ defined as $\iota_1(\alpha e_1 + v + \beta f_1) = \beta e_1 + v + \alpha f_1$, and $K_1$ the maximal compact subgroup of $G_1=\SO(V_1)$ that commutes with $\iota_1$.  Set $P_1=M_1N_1$ the parabolic subgroup of $\SO(V_1)$ that fixes the line $\R f_1$ via a right-action of $\SO(V_1)$ on $V_1$ and define $\nu: P_1 \rightarrow \GL_1$ as $f_1 p = \nu(p)^{-1} f_1$.  Let $f_1(g,s) \in Ind_{P_1}^{G_1}(|\nu|^s)$ be the $K_1$-spherical inducing section and $n_1: U \simeq N_1$ the identification of $U$ with the unipotent radical of $P_1$.  Then the intertwiner
\[\int_{U(\R)}{f_1(\iota_1 n_1(x) g,s)\,dx} = c(s) f_1(g,\dim U -s)\]
where $c(s)$ is a nonzero constant times $\frac{\Gamma(s-\dim(U)/2)}{\Gamma(s)}$.\end{lemma}
\begin{proof} Although, as mentioned, this lemma is surely well-known, we sketch a proof for the convenience of the reader.  Let $(\cdot,\cdot)_1$ denote the quadratic form on $V_1$. Define $||v||^2 = (v,\iota_1(v))_1$ for $v \in V_1$, and set $\Phi_\infty(v) = e^{-||v||^2}$ a Schwartz function on $V_1$.  Now, 
\[f_1(g,s) = \int_{\GL_1(\R)}{|t|^{s}\Phi_\infty(t(0,0,1)g)\,dt}\]
defines a $K_1$-spherical section in $Ind_{P_1(\R)}^{G_1(\R)}(|\nu|^s)$ with $f(1,s) = \Gamma(s/2)$.  Thus we can compute the $c$-function using this section $f_1(g,s)$.

We obtain
\begin{align*} \int_{U(\R)}{f_1(\iota_1 n_1(x) 1,s)\,dx} &= \int_{\GL_1(\R)}\int_{U(\R)}{|t|^{s} e^{-t^2||(1,x,-||x||^2/2)||^2}\,dt\,dx} \\ &= \Gamma(s/2) \int_{U(\R)}{\frac{dx}{(1+||x||^2/2)^s}}.\end{align*}
Thus
\begin{align*} c(s) &= \int_{U(\R)}{\frac{dx}{(1+||x||^2/2)^s}} \\ & \stackrel{\cdot}{=} \int_{0}^{\infty}{u^{\dim U/2}  (1+u)^{-s} \frac{du}{u}}\end{align*}
where the $\stackrel{\cdot}{=}$ means up to a nonzero constant.  This last integral is easily computed to be a nonzero constant times $\frac{\Gamma(s-\dim(U)/2)}{\Gamma(s)}$.\end{proof}

Applying Lemma \ref{lem:SO(N,1)}, we can now compute $M(w,s)f_{\ell,\infty}$.  For $z\in \C$ and $k\geq 0$ an integer, let $(z)_k = (z)(z+1) \cdots (z+k-1)$ denote the Pochhammer symbol.
\begin{proposition}\label{prop:longInter} Suppose 
\[V = \R e_1 \oplus \R e_2 \oplus \R e_3 \oplus U \oplus \R f_3 \oplus \R f_2 \oplus \R f_1 = \R e_1 \oplus V' \oplus \R f_1\]
with $U$ negative definite of dimension $m$ and $e_i, f_j$ isotropic with $(e_i,f_j) = \delta_{ij}$.  Denote by $P = MN$ the parabolic stabilizing $\R e_1$ for the left action of $\SO(V)$ on $V$ and $\nu:P \rightarrow \GL_1$ the character defined by $p e_1 = \nu(p) e_1$.  Suppose $w \in \SO(V)$ is defined by $w e_1 = f_1$, $w f_1 = e_1$ and $w$ is the identity of $V'$. Denote by $K$ the maximal compact subgroup of $G=\SO(V)$ that commutes with the involution $\iota$ that exchanges $e_i$ with $f_i$ and is the identity on $U$. Suppose $f_{\ell,\infty}(g,s)$ is the $K$-equivariant, $\Vm_{\ell}$-valued section in $Ind_{P(\R)}^{G(\R)}(|\nu|^{s})$ with $f_{\ell,\infty}((t,m,t^{-1}),s) = |t|^{s} x^{\ell}y^{\ell}$.  Then
\[\int_{N(\R)}{f_{\ell,\infty}(wng,s)\,dn} = c^{B_3}_{\ell}(s) f_{\ell,\infty}(g,4+m-s)\]
where 
\[c_{\ell}^{B_3}(s) = \frac{\left(\frac{s-\ell-1}{2}\right)_{\ell/2}}{\left(\frac{s-2}{2}\right)_{\ell/2+1}}  \cdot \frac{\Gamma\left(s-2 -\frac{m}{2}\right)}{\Gamma\left(s-2\right)} \cdot \frac{\left(\frac{s-2-m-\ell}{2}\right)_{\ell/2}}{\left(\frac{s-3-m}{2}\right)_{\ell/2+1}}\]
up to exponential factors and nonzero constants. Consequently, when $\ell > 2+m$, $c^{B_3}_{\ell}(s)$ is finite and $0$ at $s=\ell+1$.\end{proposition}
\begin{proof} Let $w_{12}$ denote the element of $\SO(V)$ that exchanges $e_1$ with $e_2$, $f_1$ with $f_2$ and is the identity on $\mathrm{Span}(e_1,e_2,f_1,f_2)^\perp$.  Similarly define $w_{23}$, and let $w_3$ denote the Weyl element that exchanges $e_3$ with $f_3$ is the identity on $\mathrm{Span}(e_3,f_3)^\perp$.  With this notation, the element $w$ factorizes as $w_{12} w_{23} w_{3} w_{23} w_{12}$.  

Denote by $r_1, r_2, r_3$ the absolute values of the characters of the split torus, so that 
\[r_j(\diag(t_1,t_2,t_3, 1,t_3^{-1},t_2^{-1},t_1^{-1})) = |t_j|.\]
With $P_0$ the upper-triangular minimal parabolic, we have $\delta_{P_0} = (m+4)r_1 + (m+2)r_2 + m r_3$, so that $f_{\ell,\infty}(g,s) \in Ind_P^G(\delta_{P_0}^{1/2} \lambda_s)$ with $\lambda_s = (s-2-m/2)r_1 - (1+m/2)r_2 - (m/2) r_3$. 

The intertwining operator $M(w) = M(w_{12}) M(w_{23}) M(w_3) M(w_{23}) M(w_{12})$ moves around the induction spaces as follows:
\begin{itemize}
\item $\lambda_s = (s-2-m/2)r_1 - (1+m/2)r_2 - (m/2) r_3$
\item $\stackrel{w_{12}}{\mapsto} -((2+m)/2) r_1 + (s-(4+m)/2) r_2 + -(m/2) r_3$  
\item $\stackrel{w_{23}}{\mapsto} -((2+m)/2) r_1 - (m/2) r_2 + (s-(4+m)/2) r_3$
\item $\stackrel{w_{e_3}}{\mapsto}-((2+m)/2) r_1 - (m/2) r_2 + ((4+m)/2-s) r_3$ 
\item $\stackrel{w_{23}}{\mapsto} -((2+m)/2) r_1 + ((4+m)/2-s) r_2 - (m/2) r_3$ 
\item $\stackrel{w_{12}}{\mapsto} ((4+m)/2-s)r_1 - ((2+m)/2) r_2 - (m/2) r_3$
\item $= \lambda_{4+m-s}$.
\end{itemize} 

Now applying $M(w,s)$ to the section $f_{\ell,\infty}(g,s)$ one obtains
\[M(w,s)f_{\ell,\infty}(g,s) = M(w_{12}w_{23}) \circ M(w_3) \circ M(w_{23} w_{12}) f_{\ell,\infty}(g,s).\]
Proposition \ref{prop:SL3inter} below computes the two outer intertwining operators $M(w_{12} w_{23})$ and $M(w_{23} w_{12})$.  Lemma \ref{lem:SO(N,1)} computes the inner intertwining operator $M(w_3)$.  Putting these results together gives that, up to exponential factors and nonzero constants,
\[ c^{B_3}_{\ell}(s) = \frac{\left(\frac{s-\ell-1}{2}\right)_{\ell/2}}{\left(\frac{s-2}{2}\right)_{\ell/2+1}}  \cdot \frac{\Gamma\left(s-2 -\frac{m}{2}\right)}{\Gamma\left(s-2\right)} \cdot \frac{\left(\frac{s-2-m-\ell}{2}\right)_{\ell/2}}{\left(\frac{s-3-m}{2}\right)_{\ell/2+1}}.\]
The proposition follows.\end{proof}

\begin{remark}\label{rmk:l=8} In section \ref{sec:ntm} below, we will apply Proposition \ref{prop:longInter} in the following special case: $\ell=8$ and $m = 8$. We note now that for these values, $c(s)$ is finite and nonzero at $s=9$:
\[c^{B_3}_{\ell}(s) \stackrel{\ell=8, m= 8}{=} \frac{\left(\frac{s-9}{2}\right)_{4}}{\left(\frac{s-2}{2}\right)_{5}}  \cdot \frac{\Gamma\left(s-6\right)}{\Gamma\left(s-2\right)} \cdot \frac{\left(\frac{s-18}{2}\right)_{4}}{\left(\frac{s-11}{2}\right)_{5}}.\]
\end{remark}

As used in the proof of the above proposition, we require the computation of a certain length two intertwiner of an archimedean inducing section on $\SL_3$. This computation is done in Proposition \ref{prop:SL3inter} below. To set up the proposition, let $b_1, b_2, b_3$ be the standard basis of $\R^3$, thought of as column vectors. Let $x,y$ be the standard basis of the two-dimensional representation of $\SL_2(\C)$, so that $x^2, xy, y^2$ are a basis of the $3$-dimensional representation of $K'=\SO(3)$.  We identify $b_2 + i b_3 = x^2$, $ib_1 = xy$ and $b_2-ib_3 = y^2$.  Set $f_1 = x+y$ and $f_2 = x-y$; this abuse of notation should not cause the reader any confusion.

Now suppose $\ell \geq 0$ is even and $f'_{\ell}(g,s): \SL_3(\R) \rightarrow \Vm_{\ell}$ is the section satisfying
\begin{enumerate}
\item $f'_{\ell}(gk,s) = k^{-1} \cdot f'_{\ell}(g,s)$ for all $k \in K' =\SO(3)$, $g \in \SL_3(\R)$;
\item $f'_{\ell}(pg,s) = \chi_s(p)f'_{\ell}(g,s)$, where $p = \mm{m_1}{*}{0}{m_2} \in P_{1,2}$ in $(1,2)$ block form and 
\[\chi_s(p) = |m_1|^{s} = |\det(m_2)|^{-s} = |m_1^2/\det(m_2)|^{s/3};\]
\item $f'_{\ell}(1,s) = x^{\ell}y^{\ell}$.
\end{enumerate}
Let $s_{12}$ and $s_{23}$ in $\SL_3$ be the Weyl group elements corresponding to the two simple roots, in obvious notation.  We compute the intertwiner $M(s_{23}) \circ M(s_{12}) f'_{\ell}(g,s)$.

\begin{proposition}\label{prop:SL3inter} Denote by $f''_{\ell}(g,s)$ the inducing section satisfying the first two enumerated properties above, but with $P_{12}$ replaced with $P_{21}$ and $f''_{\ell}(1,s) = f_1^{\ell}f_2^{\ell}$.  Then 
\begin{equation}\label{eqn:f''} M(s_{23}) \circ M(s_{12}) f'_{\ell}(g,s) = C_{\ell}(s) f''_{\ell}(g,s-3)\end{equation}
with 
\[C_{\ell}(s) = \frac{(\frac{s-\ell-1}{2})_{\ell/2}}{(s/2-1)_{\ell/2+1}} = \frac{ \Gamma\left(\frac{s-1}{2}\right) \Gamma\left(\frac{s}{2}-1\right)}{\Gamma\left(\frac{s-\ell-1}{2}\right) \Gamma\left(\frac{s+\ell}{2}\right)}\]
up to exponential factors and nonzero constants.\end{proposition}
\begin{proof} We begin by constructing the inducing section $f'_{\ell}(g,s)$ explicitly.  Throughout, we compute up to nonzero scalars.

Let $\Phi_{\ell}: \R^3 \rightarrow \Vm_{\ell}$ be given by $\Phi_{\ell}(v) = v^{\ell} e^{-||v||^2}$, where we consider $v \in \Vm_{1}$ and $v^{\ell}$ in the quotient $\Vm_{\ell}$ of the $\ell$-th symmetric power of $\Vm_{1}$.  Then
\[f'_{\ell}(g,s) = \frac{1}{\Gamma((s+\ell)/2)} \int_{\GL_1(\R)}{|t|^{s} \Phi_{\ell}(tg^{-1} b_1)\,dt}.\]

One checks easily that $M(s_{23}) \circ M(s_{12}) f'_{\ell}(g,s)$ is $K'$-equivariant and lands in the induction space as specified in the statement of the proposition.  Thus it suffices to compute this intertwiner when $g=1$.  This, then, is computed  by
\[\Gamma((s+\ell)/2) M(s_{23}) \circ M(s_{12}) f'_{\ell}(1,s) = \int_{\R^\times \times \R^2}{|t|^{s+\ell}(ub_1 + vb_2 + b_3)^{\ell} e^{-(u^2+v^2+1)t^2}\,dt \,du\,dv}.\]

Under the change of variables indicated above,
\begin{align*} ub_1+vb_2 + b_3 &= \frac{1}{2}\left((v-i)x^2 +2iu xy + (v+i)y^2\right) \\ &= zf_1^2 - 2i f_1 f_2 + z^* f_2^2 \end{align*}
where $z = v+iu$.  One obtains that
\[C_{\ell}(s) = \int_{\C}{\frac{(zf_1^2 - 2i f_1 f_2 + z^* f_2^2)^{\ell}}{(|z|^2+1)^{(\ell+s)/2}}\,dz}.\]

Because of the $S^1 \subseteq \C^\times$ symmetry of the domain of integration, only the coefficient of $f_1^{\ell}f_2^{\ell}$ contributes.  This coefficient is immediately seen to be
\[\sum_{0 \leq k \leq \ell/2}{\frac{\ell!}{k! k! (\ell-2k)!} z^k (z^*)^k (-2i)^{\ell-2k}}.\]
Now
\begin{align*} \int_{\C}{\frac{|z|^{2k}}{(|z|^2+1)^{(\ell+s)/2}}\,dz} &= 2\pi \int_{0}^{\infty}{\frac{r^{2k+1}}{(r^2+1)^{(\ell+s)/2}}\,dr} \\ &\stackrel{\cdot}{=} \frac{\Gamma(k+1)\Gamma((\ell+s)/2-k-1)}{\Gamma((\ell+s)/2)}\end{align*}
where the implied constant in the $\stackrel{\cdot}{=}$ is independent of $k$.  

Summing up, we have proved \eqref{eqn:f''} with
\begin{align*} C_{\ell}(s) &= \sum_{0 \leq k \leq \ell/2}{2^{\ell-2k}(-1)^{k} \frac{\ell!}{k!k!(\ell-2k)!} \frac{\Gamma(k+1)\Gamma((\ell+s)/2-k-1)}{\Gamma((\ell+s)/2)}} \\ &= \sum_{0 \leq k \leq \ell/2}{2^{\ell-2k}(-1)^{k} \frac{\ell!}{k!(\ell-2k)!} \frac{1}{\left(\frac{\ell+s}{2}-k-1\right)_{k+1}}} \\ &= \frac{1}{(s/2-1)_{\ell/2+1}} \sum_{0 \leq j \leq \ell/2}{(-4)^j \frac{\ell!}{(2j)!(\ell/2-j)!} (s/2-1)_{j}}\end{align*}
making the substitution $j = \ell/2-k$.  The proposition thus follows from the following lemma.
\end{proof}

\begin{lemma} One has
\[D_{\ell}(s):=\sum_{0 \leq j \leq \ell/2}{(-4)^j \frac{\ell!}{(2j)!(\ell/2-j)!} (s/2-1)_{j}} = 2^{\ell} \left(\frac{s-\ell-1}{2}\right)_{\ell/2}.\]
\end{lemma}
\begin{proof} First,
\[\sum_{0 \leq j \leq \ell/2}{(-4)^j \frac{\ell!}{(2j)!(\ell/2-j)!} (s/2-1)_{j}}  = (-4)^{\ell/2} \sum_{0 \leq j \leq \ell/2}{\binom{\ell}{2j} \left(\frac{1}{2}\right)_{\ell/2-j} (-1)^{\ell/2-j} (s/2-1)_j}.\]
Now
\[\binom{\ell}{2j} = \binom{\ell/2}{j} \frac{\left(\frac{1}{2}\right)_{\ell/2}}{\left(\frac{1}{2}\right)_{j} \left(\frac{1}{2}\right)_{\ell/2-j}}.\]
Moreover
\[\frac{\left(\frac{1}{2}\right)_{\ell/2} (-1)^{\ell/2-j}}{\left(\frac{1}{2}\right)_j} = \left(\frac{1-\ell}{2}\right)_{\ell/2-j}.\]
Thus
\[D_{\ell}(s) = 2^{\ell} \sum_{0 \leq j \leq \ell/2}{\binom{\ell/2}{j}(s/2-1)_j ((1-\ell)/2)_{\ell/2-j}}.\]
By the binomial property $(a+b)_n = \sum_{0 \leq k \leq n}{\binom{n}{k} (a)_k (b)_{n-k}}$ of the Pochhammer symbol, the lemma follows.
\end{proof}

\subsection{Constant term}\label{subsec:CT}  As mentioned above, we begin with the computation of the constant term of $E_{\ell}(g,\Phi_f,s)$ along $N$. For general $s$, there are three terms: $f_{\ell}(g,\Phi_f,s)$, an Eisenstein series $E_{\ell}^{M}(g,\Phi_f,s)$ on the Levi subgroup $M$, and an intertwined inducing section $M(w,s)f_{\ell}(g,\Phi_f,s)$.

We will see that at $s=\ell+1$ (in the range of absolute convergence) the intertwined inducing section $M(w,s)f_{\ell}$ vanishes and that the Eisenstein series $E_{\ell}^M(g,\Phi_f,s=\ell+1)$ is the automorphic function associated to a holomorphic weight $\ell$ modular form on $\SO(V')$.

Let us first handle the intertwining operator.  Denote by $w$ the element of $\SO(V)$ that exchanges $e$ with $f$ and is the identity on $V'$.  The intertwining operator is
\[M(w,s)f_{\ell}(g,s) = \int_{V'(\A)}{f_{\ell}(wn(x)g,s)\,dx}.\]
\begin{lemma} Suppose $\ell$ is even and $\ell > n+1$.  Then $M(w,s)f_{\ell}(g,s)$ vanishes at $s=\ell+1$.\end{lemma}
\begin{proof} The integral is absolutely convergent, so it suffices to see that the archimedean intertwiner
\[M_\infty(w,s) f_{\ell,\infty}(g,s) = \int_{V'(\R)}{f_{\ell,\infty}(wn(x)g,s)\,dx}\]
vanishes at $s= \ell+1$. This follows from Proposition \ref{prop:longInter}.
\end{proof}

The other nontrivial piece of the constant term is an Eisenstein series on the Levi subgroup $M$, $E_{\ell}^M(g,\Phi_f,s=\ell+1)$ associated to a new inducing section $f'_M(g,\Phi_f,s)$ on $M$.  This inducing section is defined as follows.  Let $b_0$ be an isotropic vector in $V'$. The section $f'_{M}(g,\Phi_f,s)$ is given by an integral
\[f_M'(g,\Phi_f,s) = \int_{((b_0)^\perp \backslash V')(\A)}{f_{\ell}(w_{0}n(x)g,s)\,dx}\]
where $f w_0 = b_0$ using the right action of $G$ on $V$.  Write $f'_{M,\infty}(g,s)$ for the corresponding archimedean inducing section, so that
\[f'_{M,\infty}(g,s) = \int_{((b_0)^\perp \backslash V')(\R)}{f_{\ell,\infty}(w_{0}n(x)g,s)\,dx}.\]
Denote by $P'$ the parabolic subgroup of $M$ that fixes $\Q b_0$ via this right action.  With $f_M'(g,\Phi_f,s)$ defined as above, $E_{\ell}^M(g,\Phi_f,s) = \sum_{\gamma \in P'(\Q)\backslash M(\Q)}{f_M'(\gamma g,\Phi_f,s)}$. 

Regarding this Eisenstein series, one has the following proposition.  For $m \in \SO(V')(\R)$, set
\[f_{hol,\ell}(m,s) = \frac{(b_0m, v_1+iv_2)^{\ell}}{||b_0 m||^{s+\ell}}.\]
Note that at $s=\ell$, $f_{hol,\ell}$ is the inducing section for the automorphic function associated to a holomorphic weight $\ell$ Eisenstein series on $\SO(V')$.
\begin{proposition}\label{prop:EisConst} Let the notation be as above.
\begin{enumerate}
\item One has 
\[f'_{M,\infty}(\diag(t,m,t^{-1}),s=\ell+1) \stackrel{\cdot}{=} |t| \left(f_{hol,\ell}(m,s=\ell)x^{2\ell} + f_{hol,\ell}(m\epsilon',s=\ell)y^{2\ell}\right).\]
\item Suppose that $\Phi_f$ is $\Q$-valued.  The automorphic function $\pi^{-\ell} E_{\ell}^M(g,\Phi_f,s=\ell+1)$ corresponds to a holomorphic modular form on $\SO(V')$ of weight $\ell$ with algebraic Fourier coefficients.
\end{enumerate}
\end{proposition}
\begin{proof} The proof of the first part follows exactly as the proof of Proposition 3.3.2 of \cite{pollackE8}.  Note that the equality here is only true at the special values of $s$ indicated; it is not true for general $s$.

Keeping track of the constants, the second part follows from the first, using the fact the absolutely convergent holomorphic Eisenstein series associated to a $\overline{\Q}$-valued inducing section has algebraic Fourier coefficients.  See, e.g., \cite{shimura} or \cite{shurman}. \end{proof}

\subsection{Rank one Fourier coefficients}\label{subsec:rank1FC} In this subsection, we prove that the rank one Fourier coefficients of $\pi^{-\ell} E_{\ell}(g,\Phi_f,s=\ell+1)$ are algebraic numbers.  As the argument and computation is identical to the calculation of the rank one Fourier coefficients of the degenerate Heisenberg Eisenstein series of \cite{pollackE8}, we are very brief.

Here is the result.
\begin{proposition}\label{prop:EisRk1} Suppose $\ell > n+1$ so that the Eisenstein series $E_{\ell}(g,\Phi_f,s)$ is absolutely convergent at $s = \ell+1$ and that the Schwartz-Bruhat function $\Phi_f$ is $\overline{\Q}$-valued.  Then the rank one Fourier coefficients of $\pi^{-\ell}E_{\ell}(g,\Phi_f,s=\ell+1)$ are $\overline{\Q}$-valued.\end{proposition}
\begin{proof} The first step of the proof is the fact that these rank one Fourier coefficients of $E_{\ell}(g,\Phi_f,s=\ell+1)$ are Euler products, given by an integral
\begin{equation}\label{eqn:etaInt}\int_{V'(\Q)\backslash V'(\A)}{\psi((\eta,x))E_{\ell}(g,\Phi_f,s=\ell+1)\,dx} = \int_{(\eta)^\perp(\A)\backslash V'(\A)}{f(\gamma_{\eta} n(x)g,\Phi_f,s=\ell+1)\,dx}\end{equation}
if $\eta \neq 0$ is isotropic.  Here $\gamma_{\eta} \in G(\Q)$ satisfies $f \gamma_\eta = \eta \in V'(\Q)$.

To prove \eqref{eqn:etaInt}, one computes that in the range of absolute convergence, the left-hand side is equal to a sum of two terms: the term appearing on the right of \eqref{eqn:etaInt} and an integral
\begin{equation}\label{eqn:etaInt2}\int_{V'(\A)}{\psi((\eta,x))f(wn(x)g,\Phi_f,s)\,dx}.\end{equation}
The content of \eqref{eqn:etaInt} is that \eqref{eqn:etaInt2} vanishes at $s=\ell+1$ if $\eta$ is isotropic.  To see this, note that the integral is absolutely convergent, so it suffices to see that the archimedean integral vanishes at $s=\ell+1$ for such an $\eta$.  This vanishing could be obtained by the arguments used to prove Theorem 3.2.5 in \cite{pollackE8}.  However, it also follows immediately from Corollary \ref{cor:FTrank2} below, so we omit this argument.

The archimedean and unramified local integrals that arise from the right-hand side of \eqref{eqn:etaInt} are computed just an is section 3.4 of \cite{pollackE8}. (In this case, the unramified finite adelic integral gives a rational number, and the archimedean integral using $f_{\ell,\infty}(g,\ell;s)$ gives the $\pi^{\ell}$.)  Finally, the finitely many ``bad'' local integrals at the finite places give algebraic numbers, as is easily seen.  This completes the proof of the proposition.
\end{proof}

\subsection{Rank two Fourier coefficients: finite part}\label{subsec:ftePart} In this subsection we do the finite part of the calculation of the rank two Fourier coefficients of our Eisenstein series.  The result of this calculation is well-known; it can be extracted from \cite{shurman}.  We briefly give the computation for the convenience of the reader. Throughout this subsection, $F$ is a local non-archimedean field with ring of integers $\mathcal{O}$, uniformizer $p$, and $|\mathcal{O}/p| = q$.

The local section for the Eisenstein series is
\[\int_{\GL_1(F)}{|t|^{s}\Phi_p(t(0,0,1)g)\,dt}.\]
Here $\Phi_p$ is a Schwartz-Bruhat function on $V(F)$ and $V(F) = Fe \oplus V' \oplus Ff$.  For the rank two Fourier coefficients, the integral that must be calculated is
\[J(s,\eta,\Phi_p):= \int_{\GL_1(F)}\int_{V'(F)}{\psi((\eta, x))\Phi_p(t(1,x,-q'(x))|t|^{s}\,dt\,dx}.\]
Here $\eta$ is a rank two element of $V'$, i.e. $(\eta,\eta) \neq 0$.  

Assume that $\Phi_p$ is unramified, i.e., that $\Phi_p$ is the characteristic function of the lattice $\mathcal{O}e \oplus V'(\mathcal{O}) \oplus \mathcal{O}f$, where $V'(\mathcal{O})$ is such that $V'(\mathcal{O}/p)$ is a non-degenerate split quadratic space over $\mathcal{O}/p$.  Breaking into pieces as determined by the valuation of $t$, one obtains
\[J(s,\eta,\Phi_p) = \sum_{r \geq 0}|p|^{rs}\left( \int_{p^{-r} V'(\mathcal{O})}{\psi((\eta,x))\charf(p^r q'(x) \in \mathcal{O})\,dx}\right).\]
In this unramified case, we will check that the terms with $r \geq 2$ vanish and will calculate the $r=1$ term explicitly.

The vanishing of the terms with $r \geq 2$ follows from the following lemma.  If $\eta \in V'(F)$, we say that $\eta$ is unramified if $\eta \in V'(\mathcal{O})$ and $q'(\eta) \in \mathcal{O}^\times$.
\begin{lemma} Suppose $r \geq 2$, $V'(\mathcal{O}/p)$ is a non-degenerate quadratic space and $\eta$ is unramified.  Then
\[\sum_{x \in V'(\mathcal{O}/p^r), q'(x) \equiv 0 \pmod{p^r}}{\psi\left(\frac{(\eta,x)}{p^r}\right)} = 0.\]
\end{lemma}
\begin{proof} The idea is to consider together all the $x$ with fixed reduction in $V'(\mathcal{O}/p^{r-1})$.  Specifically, suppose $x \in V'(\mathcal{O}/p^r)$, $q'(x) \equiv 0 \pmod{p^r}$.  Consider $x + p^{r-1}\epsilon$ for some $\epsilon \in V'(\mathcal{O})$.  Then $\frac{1}{p^r}q'(x + p^{r-1} \epsilon) = \frac{1}{p^r}(q'(x) + p^{r-1}(x,\epsilon) + p^{2r-2}q'(\epsilon))$.  As $r \geq 2$, $2r-2 \geq r$ so this is in $\mathcal{O}$ if and only if $(x,\epsilon) \in p\mathcal{O}$.  The point is that for $x \in V'(\mathcal{O}/p)$ fixed with $q'(x) \equiv 0$, there is $\epsilon$ with $(x,\epsilon) \in p \mathcal{O}$ and $(\eta,\epsilon) \in \mathcal{O}^\times$.  Indeed, if $(\eta,\epsilon) \equiv 0$ whenever $(x,\epsilon) \equiv 0$, then $x$ and $\eta$ would be $\mathcal{O}^\times$ proportional in $V'(\mathcal{O}/p)$.  But $q'(\eta) \not\equiv 0$ while $q'(x) \equiv 0$ by assumption, so such an $\epsilon$ can be found.  Perturbing the sum of those terms with reduction $x$ by $\epsilon$, one gets $0$, as desired.\end{proof}

For the $r=1$ term, we begin with the following lemma.  Let $U_n$ be the split quadratic space over $\mathcal{O}$, i.e., $U_n = \mathcal{O}^{2n}$ with quadratic form $q_n(x_1,\ldots,x_n,y_1,\ldots,y_n) = x_1 y_1 + \cdots + x_n y_n$.
\begin{lemma}\label{lem:Cn} Denote by $C(n)$ the number of elements $u$ of $U_n(\mathcal{O})/p$ with $q_n(u) \equiv 0$.  Then $C(1) = 2q-1$ and $C(n+1) = q C(n) + (q-1)q^{2n}$.\end{lemma}
\begin{proof} The formula for $C(1)$ is clear.  As for the recurrence relation, note that the elements in $U_{n+1}(\mathcal{O}/p)$ with $q_{n+1} \equiv 0$ are either of the form $(0,u_n,*)$ with $q_n(u_n) \equiv 0$ or $(*,u_n,y_1)$ with $y_1$ determined.  The recurrence follows.\end{proof}

We can now calculate the $r=1$ term in case $V'$ is split, even dimensional by induction on $n$.  Define
\[S_{n}= \sum_{v \in U_{n}(\mathcal{O}/p), q(v) \equiv 0 \pmod{p}}{\psi((\eta,v)/p)}.\]
\begin{lemma} Suppose $\eta = (1,0,\ldots, 0,1)$.  Then $S_{n+1} = - q^n$.\end{lemma}
\begin{proof} First we claim that
\begin{equation}\label{eqn:Sn1} S_{n+1} = - \sum_{v \in U_{n}(\mathcal{O}/p)}{\psi(-q(v)/p)}.\end{equation}
To see this, note that if $v \in U_{n+1}(\mathcal{O}/p)$ with $q(v) \equiv 0$, then either $v = (0,v',*)$ with $q'(v') \equiv 0$ or $v = (x_1,v',-x_1^{-1}q(v'))$ with $x_1 \in \mathcal{O}^\times$.  Summing over the first set of $v$'s gives $0$, because $\sum_{y_1 \in \mathcal{O}/p}{\psi(y_1/p)} = 0$.  Summing over the second set of $v$'s gives
\[\sum_{x_1 \in (\mathcal{O}/p)^\times, v' \in U_{n}}{\psi((x_1-x_1^{-1}q(v'))/p)}.\]
But note that 
\[\sum_{v'\in U_n(\mathcal{O}/p)}{\psi(q(v')/p)} = \sum_{v' \in U_n(\mathcal{O}/p)}{\psi(\alpha q(v')/p)}\]
for any $\alpha \in \mathcal{O}^\times$, because the split quadratic form $q$ takes all values. Thus
\[S_{n+1} = \left(\sum_{x_1 \in (\mathcal{O}/p)^\times}{\psi(x_1/p)}\right)\left(\sum_{v'\in U_n(\mathcal{O}/p)}{\psi(-q(v')/p)}\right)\]
which gives \eqref{eqn:Sn1}.

From \eqref{eqn:Sn1}, one can calculate $S_{n+1}$ in terms of $C(n)$, by breaking the sum up into those $v$ with $q(v) \equiv 0$ and those $v$ with $q(v) \not\equiv 0$.  One obtains
\[S_{n+1} = C(n) + (-1)\left(\frac{q^{2n}-C(n)}{q-1}\right) = -\frac{q}{q-1}(C(n)-q^{2n-1}).\]
But now by Lemma \ref{lem:Cn} one obtains
\[\frac{C(n+1)-q^{2n+1}}{q-1} = q \left(\frac{C(n)-q^{2n-1}}{q-1}\right)\]
so that $S_{n+1} = q S_n$.  The lemma follows.\end{proof}

Putting everything together, we arrive at the following proposition.
\begin{proposition}\label{prop:Jfte} Suppose $\dim(V)$ is even, $\eta$ is unramified, $V'(\mathcal{O}/p)$ is a non-degenerate quadratic space, and $\Phi_p$ is unramified.  Then 
\[J(s,\eta,\Phi_p) = 1 - |p|^{s-\frac{\dim V'}{2} + 1} = \frac{1}{\zeta_p(s-\dim(V')/2 + 1)}.\]
\end{proposition}
Consequently, if $\ell$ is even and $4$ divides $\dim(V')$ then the product of the unramified factors at $s=\ell+1$ gives $\pi^{-(\ell+2 - \dim(V')/2)}$ times a rational number.

We also must understand what happens at the bad finite places:
\begin{lemma} Suppose that $\Phi_p$ is $\Q$-valued, and $f(g,\Phi_p,s) = \int_{\GL_1(\Q_p)}{|t|^s\Phi_p(t(0,0,1)g)\,dt}$.  Then 
\[\int_{V'(\Q_p)}{f(w_{\ell}n(x)g,\Phi_p,s)\psi((\eta,x))\,dx}\]
is finite and valued in $\Q(\psi_p)$ at $s= n$ a postive integer.\end{lemma}
\begin{proof} First, changing $\Phi_p$ to $\Phi_p^g$, $\Phi_p^g(v) = \Phi_p(vg)$, one can assume that $g=1$.  Now, from \cite[Theorem 3.6]{karelAJM}, one obtains that there is a compact set $U$ of $V'(\Q_p)$ so that 
\begin{equation}\label{eqn:V'int} \int_{V'(\Q_p)}{f(w_{\ell}n(x),\Phi_p,s)\psi((\eta,x))\,dx} = \int_{U}{f(w_{\ell}n(x),\Phi_p,s)\psi((\eta,x))\,dx}.\end{equation}
But now, because $\Phi_p$ is Schwartz, $f(\cdot,\Phi_p,s)$ is right invariant under a compact open, so that the right-hand side of \eqref{eqn:V'int} is a finite sum.  The lemma follows because $\Phi_p$ being $R$-valued (for some ring $R$) implies $f(g,\Phi_p,s)$ is equal to $P(q^{-s})(1-q^{-s})^{-1}$ for some $R$-valued polynomial $P(X)$. \end{proof}

\subsection{Rank two Fourier coefficients: archimedean part}\label{subsec:arch} In this subsection, we calculate the archimedean contribution to the rank two Fourier coefficients of the Eisenstein series $E_{\ell}(g,\Phi_f,s)$.  More precisely, let $f_{\ell}(g,s)$ denote the $\Vm_{\ell}$-valued section of the Eisenstein series on $G$.  The main result of this subsection is the computation of the Fourier transform
\begin{equation}\label{eqn:FTint}I(\omega,\ell) = \int_{V'(\R)}{e^{-i(\omega,x)}f_\ell(wn(x),s=\ell+1)\,dx}.\end{equation}

For the involution $\iota$ on $V$ that gives rise to the Cartan involution, write $||v||^2 := (v,\iota(v))$ for the associated positive-definite norm.  Before beginning this computation, let us note that $f_{\ell}(wn(x),s)$ is the function 
\[x \mapsto \frac{\left(p_{V_3}((1,x,-q(x))\right)^{\ell}}{||(1,x,-q(x)||^{s+\ell}}.\]

Following \cite[(2.8.1)]{kOIII} define
\[\tau(x_n,x_2)^2 = \left(1 + \frac{||x_n||^2-||x_2||^2}{4}\right)^2 + ||x_2||^2.\]
One has the following simple lemma.  Denote by $p_{V_3}: V = V_3 \oplus V_{n+1} \rightarrow V_{3}$ the orthogonal projection.
\begin{lemma}\label{lem:indsec} Suppose $w \in V'(\R)$ and $w = w_2+w_n$ is the decomposition of $w$ into $V_2 \oplus V_n$, and $v = (1,w,-q'(w))$ so that $v$ is isotropic.  Then 
\[ p_{V_3}(v) = -\frac{1}{2\sqrt{2}}\left((\sqrt{2}w_2,iv_1+v_2) x^2 + (||w_2||^2-||w_n||^2-2)xy + (\sqrt{2} w_2,iv_1-v_2) y^2\right)\]
and
\[||v||^2 = \tau(\sqrt{2} w_n, \sqrt{2} w_2)^2.\]
\end{lemma}
\begin{proof} 
With notation as above, we have
\[v = \left(w_2 + \frac{1}{2}(1-q'(w))(e+f)\right) + \left(w_n + \frac{1}{2}(1+q'(w))(e-f)\right).\]
Thus
\begin{align*}
p_{V_3}(v) &= w_2 + \frac{1}{2}(1-q'(w))(e+f) \\ &= -\frac{1}{2}\left((w_2,iv_1+v_2) x^2 + \sqrt{2}(q_2(w_2)-q_n(w_n)-1)xy + (w_2,iv_1-v_2) y^2\right) \\ &= -\frac{1}{2\sqrt{2}}\left((\sqrt{2}w_2,iv_1+v_2) x^2 + (||w_2||^2-||w_n||^2-2)xy + (\sqrt{2} w_2,iv_1-v_2) y^2\right)\end{align*}
as claimed.

One computes
\begin{align*}
||(1,w,-q_{V'}(w))||^2 &= ((1,w,-q_{V'}(w)),(-q_{V'}(w),\iota(w),1)) \\ &= 1 + (q_{V'}(w))^2 + ||w_2||^2 + ||w_n||^2 \\ &= 1 + \left(\frac{||w_2||^2-||w_n||^2}{2}\right)^2 + ||w_2||^2 + ||w_n||^2 \\ &= \left(\frac{||w_2||^2-||w_n||^2}{2}-1\right)^2 + 2 ||w_2||^2 \\ &= \tau(\sqrt{2} w_n, \sqrt{2} w_2)^2.\end{align*}
This gives the lemma.\end{proof}

To compute \eqref{eqn:FTint}, we start with the answer, and compute the inverse Fourier transform.  This strategy is only possible because the unipotent group $N$ is abelian, and this is why modular forms on $G$ are much easier than modular forms on the quaternionic exceptional groups.

Thus we wish to compute
\begin{equation}\label{eqn:Ivx} I_v(x;\ell)=\int_{V'(\R)}{e^{i(\omega,x)}\charf(q(\omega) > 0) q(\omega)^{A} \left(-\frac{|(\omega,v_1+iv_2)|}{(\omega,v_1+iv_2)}\right)^{v} K_v(\sqrt{2}|(\omega,v_1+iv_2)|)\,d\omega}.\end{equation}
Eventually, we will substitute $A=\ell-n/2$.  The computations are inspired by, and use results from \cite{kOIII} and \cite{kobayashiMano}.  Compare also \cite{shimura}.

Let us first explain that the integral $I_v(x;\ell)$ is absolutely convergent if $A = \ell - n/2 \geq 0$ and $n \geq 1$. 
\begin{lemma} Suppose $A = \ell-n/2 \geq 0$ and $n \geq 1$.  Then the integral $I_v(x;\ell)$ is absolutely convergent.\end{lemma}
\begin{proof} Taking absolute values, one obtains
\begin{align*} \int_{V'(\R)} \charf(q(w) > 0) q(\omega)^{A} & K_v(\sqrt{2}|(\omega,v_1 + iv_2)|)\,d\omega \\ &= C \int_{t_2, t_n}{\charf(t_2 > t_n) (t_2^2 - t_n^2)^{A} K_v(\sqrt{2} t_2) t_2 t_n^{n-1} \,dt_2 \,dt_n} \\ &= C \int_{t_2 = 0}^{\infty}\int_{0 \leq w \leq 1}{t_2^{2A+n+1}(1-w^2)^{A} K_v(\sqrt{2} t_2) w^{n-1} \,dw\,dt_2} \\ &= C' \int_{0}^{\infty}{t_2^{2A+n+1} K_v(\sqrt{2} t_2)\,dt_2}\end{align*}
for positive constants $C, C'$.  Here we have made the variable change $t_n = w t_2$, and because $A \geq 0$ and $n \geq 1$ the integral over $w$ is finite.  Because $A = \ell - n/2$, $2A + n + 1 = 2\ell+1$.  Thus because $|v| \leq 2\ell$, $t_2^{2\ell+1} K_v(\sqrt{2} t_2)$ is $0$ at $t_2=0$ so the integral over $t_2$ in the final line above is finite.
\end{proof}
As the integral defining $I_v(x;\ell)$ is absolutely convergent, we may apply the Fourier inversion theorem, as mentioned above.

The computation of $I_v(x;\ell)$ is given in the following proposition.  We will assume $v \geq 0$ in this proposition.  Because $K_{-v}(y) = K_v(y)$, we can obtain the case $I_{v}(x;\ell)$ for $v < 0$ by the case of $v > 0$ by exchanging $v_2$ with its negative. Recall the Gauss hypergeometric function $\,_2F_1(a,b;c;z)$ and Appell's hypergeometric function $F_4(a,b;c;d;x;y)$.

\begin{proposition}\label{prop:FTint1} Suppose $v \geq 0$.  One has
\begin{align*} I_v(x;\ell) &= (2\pi)^{(n+2)/2} 2^{\ell-(n+v+2)/2} (-x_2,iv_1+v_2)^{v} \frac{\Gamma(\ell+v+1)\Gamma(\ell-n/2+1)}{\Gamma(v+1)} \\ &\times  F_4(\ell+1,\ell+1+v;\ell+1;v+1;-||x_n||^2/2;-||x_2||^2/2).\end{align*}
\end{proposition}
\begin{proof}  From \cite[(3.3.4) page 55]{kobayashiMano} (which cites \cite[Lemma 3.6, Introduction]{helgason}), one has
\begin{equation}\label{JBes1}\int_{S^{m-1}}{e^{i  \lambda (\eta,\omega)} \phi(\omega)\,d\omega} = (2\pi)^{m/2} i^{\ell} \phi(\eta) \lambda^{1-m/2} J_{m/2-1+\ell}(\lambda)\end{equation}
if $\phi$ is Harmonic of degree $\ell$, $\lambda > 0$, and $\eta$ is in the sphere $S^{m-1}$.  Here $J_{\bullet}$ is the $J$-Bessel function.

Let $S(V_n) = \{x \in V_n: ||x||^2 =1\}$ be the sphere of radius one in $V_n$, and similarly let $S(V_2)$ be the sphere of radius one in $V_2$.  We write $\omega = t_2 \sigma_2 + t_n \sigma_n$, with $t_2, t_n \in \R_{>0}$, $\sigma_2 \in S(V_2)$ and $\sigma_n \in S(V_n)$.  Let 
\[\phi_v(\omega_2) = \left(-\frac{|(\omega,v_1+iv_2)|}{(\omega,v_1+iv_2)}\right)^{v} K_v(\sqrt{2}|(\omega,v_1+iv_2)|).\]
Define $x_2 \in V_2$ and $x_n \in V_n$ so that $x = x_2 + x_n$.  Then we compute
\begin{align*} I_{v}(x;\ell) &= \int_{t_2, t_n, \sigma_2, \sigma_n}{\charf(t_2 > t_n)(t_2^2 - t_n^2)^{A} e^{it_2 (\sigma_2, x_2)}e^{i t_n (\sigma_n, x_n)}\phi(t_2 \sigma_2) t_2 t_n^{n-1}\, dt_2\, dt_n\, d\sigma_2\, d\sigma_n}
 \\ &=  (2\pi)^{n/2} \int_{t_2,t_n,\sigma_2}\charf(t_2 > t_n)(t_2^2-t_n^2)^{A} e^{it_2(\sigma_2,x_2)}\phi(t_2\sigma_2) (t_n ||x_n||)^{1-n/2} \\ &\,\,\,\, \times J_{n/2-1}(t_n ||x_n||) t_2 t_n^{n-1}\,dt_2\,dt_n\,d\sigma_2
.\end{align*}
Now by \eqref{JBes1} we have
\[\int_{\sigma_2 \in S(V_2)}{ e^{i t_2 (\sigma_2, x_2)}\phi_v(t_2\sigma_2)\,d\sigma_2} = (2\pi) \left(\frac{|(x_2,iv_1-v_2)|}{(x_2,iv_1-v_2)}\right)^{v} J_v(||x_2|| t_2) K_v(\sqrt{2} t_2).\]

Thus
\begin{align*} I_v(x;\ell) &= (2\pi)^{(n+2)/2}  \left(\frac{|(x_2,iv_1-v_2)|}{(x_2,iv_1-v_2)}\right)^{v}||x_n||^{1-n/2} \\ &\,\,\, \times \int_{t_2, t_n}{\charf(t_2 > t_n)(t_2^2-t_n^2)^{A} t_2 t_n^{n/2} K_v(\sqrt{2} t_2) J_v(||x_2|| t_2) J_{n/2-1}(||x_n|| t_n)\,dt_n \,dt_2}.\end{align*}

We now compute the integral over $t_n$.  One has
\begin{align*} \int_{0}^{t_2}{(t_2^2-t_n^2)^{A} t_n^{n/2} J_{n/2-1}(||x_n|| t_n)\,dt_n} &= (t_2)^{2A+n/2+1}\int_{0}^{1}{(1-w^2)^{A} w^{n/2} J_{n/2-1}(t_2 ||x_n|| w)\,dw} \\ &= 2^{A}\Gamma(A+1) ||x_n||^{-(A+1)} (t_2)^{A+n/2} J_{n/2+A}(||x_n|| t_2)\end{align*}
where the last line is by \cite[6.567.1]{GR}.

Combining, we obtain
\begin{align*} I_v(x;\ell) &= (2\pi)^{(n+2)/2} \left(\frac{|(x_2,iv_1-v_2)|}{(x_2,iv_1-v_2)}\right)^{v}||x_n||^{-(A+n/2)} 2^{A}\Gamma(A+1) \\ \,\,\, &\times \int_{0}^{\infty}{t_2^{A+n/2+1} J_{n/2+A}(||x_n|| t_2) J_v(||x_2|| t_2) K_v( \sqrt{2} t_2)\,dt_2}.\end{align*}
This last integral over $t_2$ is worked out in \cite[page 586]{kOIII}.  Following \emph{loc cit}, by e.g., \cite[6.578.2]{GR}, one obtains
\begin{align*}\int_{0}^{\infty}t_2^{A+n/2+1} & J_{n/2+A}(||x_n|| t_2) J_v(||x_2|| t_2) K_v( \sqrt{2} t_2)\,dt_2 = 2^{-(1+v/2)} ||x_n||^{(n/2+A)} ||x_2||^{v} \frac{\Gamma(A+n/2+1+v)}{\Gamma(v+1)} \\ & \times F_4(A+n/2+1,A+n/2+1+v;A+n/2+1;v+1;-||x_n||^2/2,-||x_2||^2/2).\end{align*}

Set $\ell = A + n/2$, and note that
\begin{equation}\label{eqn:x2arc} ||x_2|| \left(\frac{|(x_2,iv_1-v_2)|}{(x_2,iv_1-v_2)}\right) = -(x_2,iv_1+v_2).\end{equation}
Taking \eqref{eqn:x2arc} into account, the proposition follows.\end{proof}

\begin{corollary}\label{cor:FTint1} Suppose $v \geq 0$ and $||x_2|| + ||x_n|| < \sqrt{2}$.  Then 
\begin{align*} \tau(\sqrt{2}x_n,\sqrt{2} x_2)^{2\ell+1}& I_v(x;\ell) = (2\pi)^{(n+2)/2} 2^{\ell-(n+v+2)/2} (-x_2,iv_1+v_2)^{v} \frac{\Gamma(\ell+v+1)\Gamma(\ell-n/2+1)}{\Gamma(v+1)} \\ &\times  \tau(\sqrt{2}x_n,\sqrt{2}x_2)^{\ell-v} \,_2F_1\left(\frac{v-\ell}{2}, \frac{v+\ell+1}{2};v+1; \frac{2||x_2||^2}{\tau(\sqrt{2}x_n,\sqrt{2}x_2)^2}\right).\end{align*}
\end{corollary}
\begin{proof} By \cite[page 586, Lemma 5.7]{kOIII}, if $||x_n|| + ||x_2|| < \sqrt{2}$,
\begin{align*} F_4(\ell+1,\ell+1+v;\ell+1;v+1;&-||x_n||^2/2,-||x_2||^2/2) = \tau(\sqrt{2} x_2, \sqrt{2} x_n)^{-(\ell+v+1)}\\ &\,\,\, \times  \,_2F_1\left(\frac{v-\ell}{2}, \frac{\ell+v+1}{2}; \frac{v+1}{2}; \frac{2||x_2||^2}{\tau(\sqrt{2} x_2, \sqrt{2} x_n)^2}\right).\end{align*}
The corollary follows.
\end{proof}

We now restate corollary \ref{cor:FTint1} in a slightly different form.  Define
\[ J_v(x;\ell) =\frac{\Gamma(\ell+1)}{\Gamma(\ell-n/2+1)} \pi^{-(n+2)/2}\tau(\sqrt{2}x_n,\sqrt{2} x_2)^{2\ell+1} \frac{I_v(x;\ell)}{(\ell+v)!(\ell-v)!}.\]
\begin{corollary}\label{cor:Jv} For $v \geq 0$ and $||x_n|| + ||x_2|| < \sqrt{2}$, one has $J_v(x;\ell)$
\[ = 2^{\ell-v} \binom{\ell}{v} (-\sqrt{2} x_2, iv_1+v_2)^v \tau(\sqrt{2}x_n,\sqrt{2}x_2)^{\ell-v} \,_2F_1\left(\frac{v-\ell}{2}, \frac{v+\ell+1}{2};v+1; \frac{2||x_2||^2}{\tau(\sqrt{2}x_n,\sqrt{2}x_2)^2}\right).\]
For $v \leq 0$, one has
\[J_v(x;\ell) = \left(-\frac{a^*}{a}\right)^{|v|} J_{|v|}(x;\ell),\]
where $a = \sqrt{2}(x_2, iv_1+v_2)$.
\end{corollary}
\begin{proof} The first part of the corollary has already been proved.  The second follows immediately from the first by replacing $v_2$ with $-v_2$ and noting that $\sqrt{2}(x_2,iv_1-v_2) = - a^*$ if $a = \sqrt{2}(x_2, iv_1+v_2)$.\end{proof}

We now use the following lemma.  We will apply it with
\[a = -\sqrt{2} (x_2,iv_1+v_2),\,\,\, b = \left|1 + \frac{2||x_n||^2 - 2||x_2||^2}{4}\right|.\]
\begin{lemma}\label{lem:2F1Sum} For $\ell \geq 0$ even, $a \in \C^\times$ and $b >0$ with $|a| < b$, one has
\begin{align*} (ax^2 + 2bxy-a^*y^2)^{\ell} &= \sum_{0 \leq v \leq \ell}\binom{\ell}{v} 2^{\ell-v} \delta^{1/2}_{v,0}(xy)^{\ell-v}x^{2v} a^v(|a|^2+b^2)^{(\ell-v)/2} \\ &\,\,\,\,\, \times \,_2F_1\left(\frac{v-\ell}{2},\frac{v+\ell+1}{2};v+1; \frac{|a|^2}{|a|^2+b^2}\right) \\ & +\sum_{0 \leq v \leq \ell}\binom{\ell}{v} 2^{\ell-v} \delta^{1/2}_{v,0}(xy)^{\ell-v}y^{2v} (-a^*)^v(|a|^2+b^2)^{(\ell-v)/2} \\ &\,\,\,\,\, \times \,_2F_1\left(\frac{v-\ell}{2},\frac{v+\ell+1}{2};v+1; \frac{|a|^2}{|a|^2+b^2}\right).\end{align*}
Here $\delta_{v,0}^{1/2}$ is equal to $1/2$ if $v=0$ and equal to $1$ otherwise.
\end{lemma}
\begin{proof} Denote by $S_\ell^a$ the first sum on the right-hand side of the displayed equation in the statement of the lemma and $S_{\ell}^{a^*}$ the second sum on the right-hand side this equation.

One has the well-known identity
\[\,_2F_1(a',b';c';z) = (1-z)^{-a'}\,_2F_1\left(a',c'-b';c';\frac{z}{z-1}\right).\]
This follows from the integral representation
\[\,_2F_1(a',b';c';z) = \frac{\Gamma(c')}{\Gamma(b')\Gamma(c'-b')} \int_{0}^{1}{t^{b'-1}(1-t)^{c'-b'-1}(1-zt)^{-a'}\,dt},\]
valid for $Re(c') > Re(b') > 0$, by making the substitution $t \mapsto 1-t$.
Thus
\[(|a|^2+b^2)^{(\ell-v)/2}\,_2F_1\left(\frac{v-\ell}{2},\frac{v+\ell+1}{2};v+1; \frac{|a|^2}{|a|^2+b^2}\right) = b^{\ell-v}\,_2F_1\left(\frac{v-\ell}{2},\frac{v-\ell}{2} + \frac{1}{2};v+1;-\frac{|a|^2}{b^2}\right).\]

We thus must evaluate the sum
\[\sum_{0 \leq v \leq \ell}{\binom{\ell}{v}(2bxy)^{\ell-v} a^{v} \delta^{1/2}_{v,0} x^{2v}\,_2F_1\left(\frac{v-\ell}{2}, \frac{v-\ell}{2}+\frac{1}{2};v+1; -\frac{|a|^2}{b^2}\right)}\]

Now, note that if $\delta \geq 0$ is an integer then $\left(-\frac{\delta}{2}\right)_m \left(-\frac{\delta}{2}+ \frac{1}{2}\right)_m = \left(\frac{1}{4}\right)^m \frac{\delta!}{(\delta-2m)!}$.  Thus, plugging in the definition of $\,_2F_1$, we obtain
\begin{align*} S_\ell^a &= \sum_{0 \leq v \leq \ell, 0 \leq m}\binom{\ell}{v}(2bxy)^{\ell-v} a^{v} \delta^{1/2}_{v,0} x^{2v}(-1)^m \frac{ \binom{\ell-v}{2m} (2m)!}{(v+1)_m m!} \left(\frac{|a|}{2b}\right)^{2m} \\ &= \sum_{0 \leq v \leq \ell, 0 \leq m} \frac{\ell !}{(\ell-v-2m)! (v+m)! m!} (-1)^m (2bxy)^{\ell-v-2m} a^v |a|^{2m}(xy)^{2m} \delta^{1/2}_{v,0} x^{2v}. \end{align*}
Similarly,
\[S_{\ell}^{a^*} = \sum_{0 \leq v \leq \ell, 0 \leq m} \frac{\ell !}{(\ell-v-2m)! (v+m)! m!} (-1)^m (2bxy)^{\ell-v-2m} (-a^*)^v |a|^{2m}(xy)^{2m} \delta^{1/2}_{v,0} y^{2v}.\]
The lemma now follows easily.
\end{proof}

The condition $|a| < b$ in Lemma \ref{lem:2F1Sum} can now be removed:
\begin{corollary}\label{cor:2F1Sum} The statement of Lemma \ref{lem:2F1Sum} holds under the condition $b > 0$, and not just $b > |a|$.\end{corollary}
\begin{proof} Suppose $b > 0$ and set $t = a/b$, so that $t \in \C$.  Dividing both sides of the statement of Lemma \ref{lem:2F1Sum} by $b^{\ell}$, one obtains
\begin{align*} (tx^2 + 2xy - t^* y^2)^{\ell} &= \sum_{0 \leq v \leq \ell}\binom{\ell}{v} 2^{\ell-v} \delta^{1/2}_{v,0}(xy)^{\ell-v}x^{2v} t^v(|t|^2+1)^{(\ell-v)/2} \\ &\,\,\,\,\, \times \,_2F_1\left(\frac{v-\ell}{2},\frac{v+\ell+1}{2};v+1; \frac{|t|^2}{|t|^2+1}\right) \\ & +\sum_{0 \leq v \leq \ell}\binom{\ell}{v} 2^{\ell-v} \delta^{1/2}_{v,0}(xy)^{\ell-v}y^{2v} (-t^*)^v(|t|^2+1)^{(\ell-v)/2} \\ &\,\,\,\,\, \times \,_2F_1\left(\frac{v-\ell}{2},\frac{v+\ell+1}{2};v+1; \frac{|t|^2}{|t|^2+1}\right).\end{align*}
From Lemma \ref{lem:2F1Sum}, the above equality holds for $t \in \C$ with $|t| < 1$.  However, both sides are analytic functions of $t$, so their equality for $|t| < 1$ implies their equality for all $t \in \C$.  The corollary follows.\end{proof}

Combining Corollary \ref{cor:2F1Sum} with Corollary \ref{cor:Jv}, we obtain:
\begin{corollary}\label{cor:FTrank2} Set $a_1 = \sqrt{2}(x_2, iv_1+v_2)$ and $b_1 = (||x_2||^2-||x_n||^2-2)/2$.  Then
\begin{equation}\label{eqn:FTrank2}\frac{\Gamma(\ell+1)}{\Gamma(\ell-n/2+1)} \pi^{-(n+2)/2} \sum_{-\ell \leq v \leq \ell}{I_v(x;\ell)\frac{x^{\ell+v}y^{\ell-v}}{(\ell+v)!(\ell-v)!}} = \frac{(a_1 x^2 + 2b_1xy - a_1^*y^2)^{\ell}}{\tau(\sqrt{2} x_n,\sqrt{2} x_2)^{2\ell+1}}.\end{equation}
\end{corollary}
\begin{proof} First suppose $||x_2|| + ||x_n|| < \sqrt{2}$.  Note that the assumption $||x_2|| + ||x_n|| < \sqrt{2}$ implies that $b_1$ is negative, and thus $b = |b_1| = -b_1$.  Therefore, from Corollaries \ref{cor:2F1Sum} and \ref{cor:Jv} the equality above holds so long as $||x_2|| \neq 0$, $||x_n|| \neq 0$, and $b = |1 + \frac{||x_n||^2-||x_2||^2}{2}| \neq 0$.  The conditions $||x_2|| \neq 0$ and $||x_n|| \neq 0$ are used in the manipulations used to prove Proposition \ref{prop:FTint1}.  But now the absolute convergence computations for $I_v(x;\ell)$ prove that $I_v(x;\ell)$ is a continuous function of $x$.  As both sides of \eqref{eqn:FTrank2} are continuous in $x$, the corollary follows in this case.

For the general case, it follows from Proposition \ref{prop:FTint1} that $I_v(x;\ell)$ is an analytic function of $x_2$ and $||x_n||^2$.  Therefore, the equality \eqref{eqn:FTrank2} for $||x_2|| + ||x_n|| < \sqrt{2}$ implies the equality for all $x_2, x_n$. \end{proof}

We arrive at the main archimedean theorem of this subsection.
\begin{theorem}\label{thm:FTAinv} The Fourier transform
\[\int_{V'(\R)}{e^{-2\pi i(\omega,x)} f_{\ell}(wn(x),s=\ell+1)\,dx} = C_{\ell,n} (2\pi)^{2\ell+2 - \dim(V')/2} q(\omega)^{\ell-n/2} \Wh_{2\pi \omega}(1)\]
for a nonzero rational number $C_{\ell,n}$.
\end{theorem}
\begin{proof} Making the change of variable $\omega \mapsto 2\pi \omega$ in the integral $I_v(x;\ell)$ of \eqref{eqn:Ivx}, one gets
\begin{equation}\label{eqn:chOv}(2\pi)^{-(2\ell + 2)} I_v(x;\ell) = \int_{V'(\R)}{e^{2\pi i (\omega,x)} q(\omega)^{\ell-n/2} \left(-\frac{|(2\pi\omega,v_1+iv_2)|}{(2\pi \omega,v_1+iv_2)}\right)^{v} K_v(\sqrt{2}|(2\pi\omega,v_1+iv_2)|)\,d\omega}.\end{equation}

From Corollary \ref{cor:FTrank2}, one obtains
\begin{equation}\label{eqn:FTrank3} \pi^{-(2\ell+2)}  \sum_{-\ell \leq v \leq \ell}{I_v(x;\ell)\frac{x^{\ell+v}y^{\ell-v}}{(\ell+v)!(\ell-v)!}} \stackrel{\cdot}{=} \pi^{(n+2)/2 - (2\ell+2)}\frac{(a_1 x^2 + 2b_1xy - a_1^*y^2)^{\ell}}{\tau(\sqrt{2} x_n,\sqrt{2} x_2)^{2\ell+1}}\end{equation}
in the notation of that corollary, where the $\stackrel{\cdot}{=}$ means that the two sides are equal up to a nonzero rational number.  By applying Lemma \ref{lem:indsec}, the right-hand side of \eqref{eqn:FTrank3} is $\pi^{(n+2)/2 - (2\ell+2)} f_{\ell}(wn(x),s=\ell+1)$.  Thus from \eqref{eqn:chOv}, one gets
\[\pi^{2\ell+2 - \dim(V')/2} \int_{V'(\R)}{e^{2\pi i (\omega,x)} q(\omega)^{\ell-n/2} \Wh_{2\pi \omega}(1)\,d\omega} \stackrel{\cdot}{=} f_{\ell}(wn(x),s=\ell+1).\]
The theorem follows by Fourier inversion. \end{proof}

Combining Theorem \ref{thm:FTAinv} with Proposition \ref{prop:EisConst} and Proposition \ref{prop:EisRk1} proves the first main theorem of this paper:
\begin{theorem}\label{thm:EisAlgFC} Suppose $\ell > n+1$ is even, $\dim(V')$ is a multiple of $4$ and $\Phi_f$ is valued in $\overline{\Q}$.  Then the Fourier coefficients of the Eisenstein series $\pi^{-\ell} E(g,\Phi_f,\ell;s)$ at $s=\ell+1$ are algebraic numbers.\end{theorem}

\section{Constant terms}\label{sec:Const} In this section, we show that the constant terms of modular forms on $\SO(4,n+2)$ (in the sense of \cite{pollackQDS}) to $\SO(3,n+1)$ are modular forms in the sense of section \ref{subsec:MFdef}.  Moreover, the quaternionic exceptional groups of type $F_4, E_6, E_7, E_8$ have Levi subgroups $L$ of type $B_{3,3}, D_{4,3}, \SU(2) \times D_{5,3}$ and $D_{7,3}$ respectively.  We also check that the constant terms of modular forms on these exceptional groups to the above $L$ are modular forms on $L$.  More precisely, in section \ref{sec:MFFE} we defined modular forms on groups $\SO(V)$, but the definition extends immediately to groups $L$ isogenous to these $\SO(V)$, which is what occurs for the above exceptional groups.

\subsection{Orthogonal groups of rank four}\label{subsec:OrthConst}  Let $V_{4,n+2} = H \oplus V$ be the rational quadratic space that is the orthogonal direct sum of $V$ and a hyperbolic plane.  The signature of $V$ is $(4,n+2)$.  Denote by $e_0, f_0$ a standard basis of the hyperbolic plane, so that the pairing $(e_0,f_0) = 1$.  Denote by $P_0 = M_0 N_0$ the parabolic subgroup of $\SO(V_{4,n+2})$ that stabilizes the line spanned by $e_0$ for the left action of $\SO(V_{4,n+2})$ on $V_{4,n+2}$.  The Levi subgroup $M_0$ is defined to be the one that stabilizes both $\mathrm{Span}(e_0)$ and $\mathrm{Span}(f_0)$.  Extend the involution $\iota$ on $V$ to $V_{4,n+2}$ by defining it to exchange $e_0$ and $f_0$.  We take as a Cartan involution on $\SO(V_{4,n+2})(\R)$ conjugation by $\iota$.  Let $K_{4,n+2}$ be the maximal compact subgroup that is the fixed points of this involution.

In \cite{pollackQDS} we defined and considered modular forms on the group $\SO(V_{4,n+2})$.  If $\varphi$ is a modular form of weight $\ell \geq 1$ on $\SO(V_{4,n+2})$, one can take the constant term of $\varphi$ along $N_0$ to obtain at automorphic function $\varphi_{N_0}$ on $M_0$.  The purpose of this subsection is to prove that this constant term is a modular form on $M_0$ of weight $\ell$, in the sense of section \ref{sec:MFFE}.  This fact follows immediately from the following proposition.

To setup the proposition precisely and to prove it, we make a few notations. Let $\{X_\gamma\}_\gamma$ be a basis of $\p=\p_{3,n+1}=V_3 \otimes V_{n+1}$, $\{u_1, \ldots, u_n, u_{n+1}\}$ be a basis of $V_{n+1}$, and $w_1, w_2, w_3$ a basis of $V_3$.  Write $y_+ = e_0+ f_0$ and $y_- = e_0 - f_0$.  Recall the operator $\widetilde{D}_\ell$ from subsection \ref{subsec:MFdef}, and the analogous operator from \cite[subsection 7.1]{pollackQDS}.  To distinguish these two operators, we write $\widetilde{D}_{4,n+2}$ for the one that acts on $\Vm_{\ell}$-valued automorphic functions on $\SO(V_{4,n+2})$ and similarly $\widetilde{D}_{3,n+1}$ for the one that acts on $\Vm_{\ell}$-valued automorhpic functions on $\SO(V)$.  Analogously, we write $D_{4,n+2}$, respectively $D_{3,n+1}$, for the so-called Schmid operators, which by definition are the $\widetilde{D}$'s followed by an appropriate $\SU(2)$-contraction $\pr_{-}:Y_2 \otimes Sym^{2\ell}(Y_2) \rightarrow Sym^{2\ell-1}(Y_2)$.

\begin{proposition}\label{prop:SO4const} Let the notation be as above.  Suppose $\varphi': \SO(V_{4,n+2})(\R) \rightarrow \Vm_{\ell}$ is left $N_0(\R)$-invariant, $\varphi'(gk) = k^{-1} \cdot \varphi'(g)$ for all $k \in K_{4,n+2}$ and $g \in \SO(V_{4,n+2})(\R)$, and $D_{4,n+2}\varphi'(g) = 0$.  Denote by $\varphi_{M_0}$ the restriction of $\varphi'$ to $\SO(V)(\R) \subseteq M_0(\R)$.  Then $D_{3,n+1}\varphi_{M_0}= 0$.\end{proposition}
\begin{proof} The proof follows without much difficulty, directly from the definitions.

With the above notation, we have
\begin{align*} \widetilde{D}_{4,n+2} \varphi' &= \sum_{\gamma }{X_\gamma \varphi' \otimes X_\gamma^\vee} +\sum_{1 \leq j \leq n+1}{(y_+ \wedge u_j)\varphi' \otimes (y_+ \wedge u_j)^\vee} \\ &\, + (y_+ \wedge y_{-})\varphi' \otimes (y_+ \wedge y_{-})^\vee +  \sum_{1 \leq k \leq 3}{(w_k \wedge y_{-})\varphi' \otimes (w_k \wedge y_{-})^\vee}.\end{align*}

Note that restricting to $M_0$ and applying $\pr_{-}$ to the first term gives $D_{3,n+1}\varphi_{M_0}$.  Moreover, restricting the second term to $M_0$ gives $0$ because $y_{+} \wedge u_j = 2 e_0 \wedge u_j - y_{-} \wedge u_j \in \n_0 \oplus Lie(\SO(n+2))$, and $\varphi'$ is invariant on the left under $N_0$ and on the right under $\SO(n+2)$.  Thus we obtain
\[D_{4,n+2}\varphi_{M_0} = D_{3,n+1}\varphi_{M_0} + \left.\pr_{-}\left(\sum_{1 \leq j \leq 4}{(h_j \wedge y_{-})\varphi' \otimes (h_j \wedge y_{-})^\vee}\right)\right|_{M_0}\]
where $\{h_1,h_2,h_3,h_4\}=\{y_{+},w_1,w_2,w_3\}$ is a basis of the four-dimensional space $V_{4,n+2}^{\iota = 1}.$  Note that the $\pr_{-}$-term is linearly independent from the $D_{3,n+1}$ term, because it contains a $y_{-}^\vee$.  This proves that $D_{3,n+1}\varphi_{M_0} = 0$, as desired.
\end{proof}

\subsection{Exceptional groups}\label{subsec:ExcConst} Suppose $C$ is a rational composition algebra, with $C \otimes \R$ positive definite.  Set $J = H_3(C)$ and $G_J$ the quaternionic exceptional group associated to $J$ as in \cite{pollackQDS}.  Thus $G_J$ has rational root type $F_4$ and is of Dynkin type $F_4, E_6, E_7$ or $E_8$ depending on if $\dim C$ is $1,2,4$ or $8$.  Let $Q_J =L_J V_J$ be the standard maximal parabolic subgroup of $G_J$ with Levi subgroup $L_J$ of rational type $B_3$.  In this subsection, we prove that the constant term $\varphi_{V_J}$ of a modular form $\varphi$ of weight $\ell$ on $G_J$ is a modular form of weight $\ell$ on $L_J$.  Moreover, we prove that the rank one and rank two Fourier coefficients of $\varphi_{V_J}$ are the rank one and rank two Fourier coefficients of $\varphi$.

To state precisely these results and prove them, we now make some definitions.  Let the simple roots of $F_4$ be $\alpha_j$ with $1 \leq j \leq 4$.  We label the simple roots so that $\alpha_j$ is connected to $\alpha_{j+1}$ in the Dynkin diagram, for $j =1,2,3$, and with $\alpha_1, \alpha_2$ the long roots:
\[ \circ ---- \circ ==>== \circ ---- \circ;\]
the roots are labeled $1,2,3,4$ from left to right.  Write a positive root as a four-tuple $[n_1 n_2 n_3 n_4]$, which corresponds to $\sum_j{n_j \alpha_j}$.  The rational root spaces corresponding to long roots of $F_4$ are one-dimensional while the rational root spaces corresponding to short roots spaces of $F_4$ can be identified with the composition algebra $C$.

The Heisenberg parabolic of $G_J$ that is central to \cite{pollackQDS} is the maximal parabolic with simple root $\alpha_1$ in its unipotent radical.  We define $Q_J = L_J V_J$ to be the standard maximal parabolic subgroup of $G_J$ with simple root $\alpha_4$ in its unipotent radical $V_J$.  Thus $L_J$ has rational root type $B_3$.  The parabolic subgroup $Q_J$ of $G_J$ defines a $5$-step $\Z$-grading on the Lie algebra $\g(J) = Lie(G_J)$.  Specifically, for $j = -2,-1,1,2$, set $V_J^{j}$ the subspace of $\g(J)$ consisting of those rational roots spaces $[n_1,n_2,n_3,n_4]$ where $n_4 = j$.  Then $V_J^{\pm 2}$ are each a direct sum of $6$ long root spaces and one short root space, while $V_J^{\pm 1}$ is a direct sum of $8$ short root spaces.  One has a direct sum decomposition
\[V_{J}^{-2} \oplus V_{J}^{-1} \oplus Lie(L_J) \oplus V_J^{1} \oplus V_J^{2}.\]
See also \cite[section 2]{savinWeissman} for more on this Lie algebra decomposition.

As mentioned, the group $L_J$ is, up to anisotropic factors, $\SO(H^3 \oplus C) = \SO(V_J^{2})$.  The map $L_J \rightarrow O(V_J^{2})$ is induced by an $L_J$-invariant rational quadratic form on $V_J^{2}$.  Such a non-degenerate rational symmetric form $(\,,\,)_J$ on $V_J^{2}$ can be defined as follows.  Denote by $\alpha_{L_J} = \alpha_1 + 2 \alpha_2 + 3 \alpha_3 + 2\alpha_4$, a root of $F_4$.  Let $s_{L_J}$ be the Weyl group element that is reflection in $\alpha_J$.  Abusing notation, denote also by $s_{L_J}$ an element of the group $G_J(\Q)$ that represents $s_{L_J}$ in the Weyl group.  Denote by $B_{\g(J)}(\cdot,\cdot)$ the multiple of the Killing form on $\g(J)$, normalized as in \cite[section 4.2.2]{pollackQDS}.  Then for $v,w \in V_J^{2}$, set $(v,w)_J = B_{\g(J)}(v, s_{L_J} w)$.  Because $s_{L_J}$ is an involution and in $G_J$, the pairing $(\,,\,)_J$ is symmetric.  It is easily checked to non-degenerate.  Finally, the reflection $s_{L_J}$ acts as the identity on $Lie(L_J)$, from which one concludes that $(\,,\,)_J$ is $L_J$-invariant.

That the constant term $\varphi_{V_J}$ is a modular form of weight $\ell$ on $L_J$ follows immediately from the following proposition.  Similar to subsection \ref{subsec:OrthConst}, let $\widetilde{D}_J$, $D_J = \pr_{-} \circ \widetilde{D}_J$ denote the differential operators used to define modular forms on $G_J$ and $\widetilde{D}$, $D = \pr_{-} \circ \widetilde{D}$ denote the differential operators used to define modular forms on $L_J$.
\begin{proposition}\label{prop:GJconst} Let the notation be as above.  Suppose $\varphi': G_J(\R) \rightarrow \Vm_{\ell}$ is left $V_J(\R)$-invariant, $\varphi'(gk) = k^{-1} \varphi'(g)$ for all $k \in K_{J}$ and $g \in G_J(\R)$, and $D_{J}\varphi'(g) = 0$.  Denote by $\varphi_{L_J}$ the restriction of $\varphi'$ to $L_J(\R)$.  Then $D \varphi_{L_J}= 0$.\end{proposition}
\begin{proof} Set $L'$ to be a subgroup of $G_J(\R)$ that contains $L_J(\R)$ and has Lie algebra $V_J^{-2} \oplus Lie(L_J) \oplus V_{J}^{2}$.  Denote by $Q' = L' \cap Q_J(\R)$; $Q'$ is a parabolic subgroup of $L'$ with Levi subgroup $L_J(\R)$.  Then $L'$ has real root type $B_4$; in particular, it is isogenous to $\SO(4,4+\dim(C))$.

The idea of the proof is simple.  We check that $\varphi'$ restricted to $L'$ satisfies the assumptions of the $\varphi'$ of Proposition \ref{prop:SO4const}.  That is, denoting $\varphi'' := \varphi'|_{L'}$, we check that $D_{4,4+\dim(C)}\varphi'' = 0$.  The other assumptions on the $\varphi'$ of Proposition \ref{prop:SO4const} are immediately verified.  Then, applying Proposition \ref{prop:SO4const} to $\varphi''$, one concludes that 
\[ \varphi''|_{L_J} = \left(\varphi'|_{L'}\right)|_{L_J} = \varphi_{L_J}\]
is annihilated by $D$.

So, it remains to check that $D_{4,4+\dim(C)}\varphi'' = 0$.  We use the notation of \cite[section 6.3]{pollackQDS}.  For an element $r$ in the composition algebra $C$, and an integer $j \in \{1,2,3\}$ let $x_j(r)$ be the corresponding element of $H_3(C) = J$.  Note that if $v \in \mathrm{Span}\{v_1,v_2,v_3\}$, so that $v \otimes x_j(r) \in \g(J)$, then
\[\frac{1}{2}\left( v \otimes x_j(r) + \iota(v) \otimes x_j(r)\right) = n_v(x_j(r))\]
is in the compact Lie algebra\footnote{The author apologizes for the similar-looking notations $L_J$ and $L^0(J)$.  The group $L_J$ is the Levi of a maximal parabolic subgroup of $G_J$, while $L^0(J) \subseteq G_J(\R)$ is compact.  In case $J=H_3(\Theta)$ so that $G_J$ is of type $E_8$, $L_J$ is a rational reductive group $\GSpin(V_J^{2})$, while $L_J^0$ is a compact form of $E_7$.} $Lie(L^0(J))$.  Again, see \cite[section 6.3]{pollackQDS}.  

The vector space $V_J^{1}$ of $\g(J)$ is spanned by elements of the form $v \otimes x_j(r)$ and $\delta \otimes x_j(r)$.  Suppose $X \in V_1$ so that $X - \theta(X) \in \p_J$.  We have $X - \theta(X) = 2X - (X + \theta(X))$, with $X + \theta(X) \in Lie(L^0(J))$.  Consequently, if $g \in Q'$, then $\left((X-\theta(X))\varphi'\right)(g) = 0$ because $\varphi'$ is left $V_J(\R)$-invariant and right $L_J^{0}$-invariant.   It follows formally that one has an equality $D_J \varphi'(g) = D_{4,4+\dim(C)} \varphi'(g)$.  As $D_J\varphi' = 0$ by assumption, the proposition follows.
\end{proof}

Finally, we end this subsection by comparing the Fourier coefficients of a modular form of weight $\ell$ on $G_J$ to those of its constant term along $V_J$.
\begin{proposition}\label{prop:CHrk1rk2} Suppose $\varphi$ is a modular form of weight $\ell$ on $G_J$ with constant term $\varphi_{V_J}$ to $L_J$; $\varphi_{V_J}$ is a modular form of weight $\ell$ on $L_J$.  The rank one Fourier coefficients of $\varphi$ can be identified with rank one Fourier coefficients of $\varphi_{V_J}$, and similarly the rank two Fourier coefficients of $\varphi$ can be identified with the rank two Fourier coefficients of $\varphi_{V_J}$.  In particular, if $\varphi_{V_J}$ has Fourier coefficients in $E$ for some field $E \subseteq \C$, then $\varphi$ has rank one and rank two Fourier coefficients in $E$.\end{proposition}

Before giving the proof of Proposition \ref{prop:CHrk1rk2}, let us note that it is not a statement that has an analogue for general automorphic functions; although, the analogous statement \emph{is} true for holomorphic Siegel modular forms. Rather, the truth of Proposition \ref{prop:CHrk1rk2} crucially uses the robust Fourier expansion (\cite[Corollary 1.2.3]{pollackQDS}) of modular forms on $G_J$ and the Fourier expansion of modular forms on $L_J$ (Theorem \ref{thm:IntroFE}).  Specifically, the crux of the matter is that the generalized Whittaker functions $\Wh_\omega(g)$ of \cite[Theorem 1.2.1]{pollackQDS} have the following extra invariance property: 
\begin{equation}\label{eqn:WhitInv} \omega \in W_J(\R), m \in H_J^{1}(\R), \text{ and } m \cdot \omega = \omega \text{ implies } \Wh_{\omega}(mg) = \Wh_\omega(g).\end{equation}
Here the notation is from \emph{loc cit} so that $H_J^{1}$ is the similitude-equal-one part of the Levi of the Heisenberg parabolic $P_J = H_J N_J$ of $G_J$ and $W_J = N_J/[N_J,N_J]$ is the defining representation of $H_J$.

\begin{proof}[Proof of Proposition \ref{prop:CHrk1rk2}] Identify $W_J$ with the degree $1$ part of the $5$-step $\Z$-grading on $\g(J)$ determined by the Heisenberg parabolic $P_J$ of $G_J$.  Now, set $W'_0 = W_J \cap Lie(L_J)$ and $W'_2 = W_J \cap V_J^2$.  Then $W'_0$ and $W'_2$ are paired nontrivially under the symplectic form $\langle \cdot, \cdot \rangle$ on $W_J$ and both can be identified with $H^2 \oplus C$ for a hyperbolic plane $H$.

Now, suppose $\omega \in W_J(\Q)$ has rank at most two, and we wish to compute the Fourier coefficient 
\[\varphi_\omega(g) = \int_{N_J(\Q)\backslash N_J(\A)}{\psi^{-1}(\langle \omega, \overline{n}\rangle) \varphi(ng)\,dn}.\]
Here $\overline{n}$ denotes the image of $n$ in $W_J$.  Because $\omega$ has rank at most two, there is $m \in H_J^1(\Q)$ so that $m \omega \in V'_2$.  Consequently, by automorphy of $\varphi$, we can assume $\omega \in V'_2$.

The Levi subgroup $L_J$ is, up to compact factors, isogenous to $\SO(H \oplus V'_0)$ for a hyperbolic plane $H$.  The $\omega \in W'_2$ determines an $\eta \in W'_0$ so that the Fourier coefficient $\varphi_{V_J,\eta}$ of $\varphi_{V_J}$ defined by $\eta$ can be written as an integral of $\varphi_\omega$.  Specifically, one has formally an equality
\begin{equation}\label{eqn:VJeta} \varphi_{V_J,\eta}(g) = \int_{(V_J \cap H_J)(\Q)\backslash (V_J \cap H_J)(\A)}{\varphi_\omega(ng)\,dn}.\end{equation}
But the elements of $V_J \cap H_J$ have similitude equal to $1$ and act as the identity on $W'_2$.  Consequently, applying Corollary 1.2.3 of \cite{pollackQDS} and the invariance property \eqref{eqn:WhitInv}, the integral in \eqref{eqn:VJeta} becomes $\varphi_{\omega}(g)$ times a volume of a compact subgroup of $G_J(\A_f)$.

Comparing the generalized Whittaker functions of Definition \ref{def:Weta} with those of \cite[Theorem 1.2.1]{pollackQDS}, and taking note of the rational quadratic form $(\,,\,)_J$ on $V_J^2$ defined above Proposition \ref{prop:GJconst}, one obtains the proposition. \end{proof}

\section{The next-to-minimal modular form}\label{sec:ntm} In this section, we give an application of all of the above results, and prove that the so-called next-to-minimal modular form on $E_{8,4}$ has rational Fourier coefficients, under a mild assumption.  These next-to-minimal representations and some results about their Fourier coefficients have appeared in \cite{kobayashiSavin}, \cite{GGKPS}.

More precisely, we prove the following result.  Let $\Theta$ be the positive definite octonions, $J = H_3(\Theta)$, and $E_J(g,s;n) = \sum_{\gamma \in P_J(\Q)\backslash G_J(\Q)}{f_J(\gamma g,s,n)}$ the Eisenstein series of \cite{pollackE8} with \emph{spherical} inducing data at every finite place, normalized so that the inducing section $f_J(g,s,n)$ takes the value $\frac{\zeta(n+1)}{(2\pi)^n}$ at $s=n+1$, $g=1$.
\begin{theorem}\label{thm:ntmRat} The Eisenstein series $E_J(g,s;8)$ is regular at $s=9$ and defines a square integrable modular form of weight $8$ at this point.  Its rank zero, rank one, and rank two Fourier coefficients are all rational numbers.  In particular, if $E_J(g,s=9;8)$ has vanishing rank three and rank four Fourier coefficients (which happens if Property $V$ below holds), then $E_J(g,s;8)$ has rational Fourier expansion.\end{theorem}

The value $E_J(g,s=9;8)$ is expected to be the ``next-to-minimal'' modular form on $E_{8,4}$, and as such, should have vanishing rank three and rank four Fourier coefficients.  In fact, this vanishing would follow from the analogous local statement, for one finite prime $p$:\\

\noindent \textbf{Property V}: Denote by $\Pi$ the spherical constituent of the induced representation $Ind_{P_J(\Q_p)}^{G_J(\Q_p)}(\delta_P^{s_1})$, for $s_1 = \frac{20}{29} = \frac{1}{2} + \frac{11}{58}$.  Then the twisted Jacquet module $\Pi_{N_J,\chi}$ is $0$ for every unitary character $\chi$ of $N_J$ that is rank three or rank four.\\

Theorem \ref{thm:ntmRat} is the analogue for the ``next-to-minimal" modular form on quaternionic $E_8$ of results proved about the minimal modular form in \cite{ganATM} and \cite{pollackE8}.  The proof of Theorem \ref{thm:ntmRat} consists of a few steps, which we now outline.  
\begin{enumerate}
\item First, we analyze a certain spherical Eisenstein series $E_{V',8}(g,s)$ on the group $\SO(H^2 \oplus \Theta)$, which has signature $(2,10)$.  With our normalizations, the point $s=8$ is outside the range of absolute convergence for this Eisenstein series, but we check that at $s=8$ $E_{V',8}(g,s)$ is regular and defines a holomorphic modular form.  Moreover, this Eisenstein series has rational Fourier coefficients.  These facts are likely well-known, but we briefly prove them because the author is unaware of a suitable reference.
\item Second, we analyze the spherical Eisenstein series $E_{8}(g,s=9)$ on the group $\SO(H^3 \oplus \Theta)$.  The point $s=9$ is outside the range of absolute convergence, but by doing the appropriate intertwining operator calculations and using the first step, one can show that $E_8(g,s)$ is regular at $s=9$ and defines a modular form of weight $8$ at this point.  Moreover, via calculations as above, we know that $\pi^{-8}E_{8}(g,s=9)$ has rational Fourier coefficients.
\item Thirdly, we show that the Eisenstein series $E_J(g,s;8)$ is regular at $s=9$ and defines a modular form of weight $8$ at this point.  The proof proceeds similarly to the proof of Corollary 4.1.2 of \cite{pollackE8}.  In particular, by doing many intertwining operator calculations, we compute the constant term of $E_J(g,s;8)$ at $s=9$ along the minimal parabolic $P_0$ of $G_J$.  From this calculation, we deduce that the differential operator $D_8$ annihilates the constant term of $E_J(g,s=9;8)$ and then consequently the entire Eisenstein series: $D_8 E_J(g;s=9;8) = 0$.
\item Fourthly, we prove that the constant term $E_J^{V}(g;s=9;8)$ of $E_J(g;s=9;9)$ to the $D_{7,3}$ Levi subgroup is the Eisenstein series $E_{8}(g,s=9)$.  To do this, one considers the difference $E_J^{V}(g;s=9;8) - E_{8}(g,s=9)$, and shows uses the third step that this difference has constant term $0$ to minimal parabolic $P_0$.  It follows easily from this fact that $E_J^{V}(g;s=9,8) = E_{8}(g,s=9)$.
\item The constant term of $E_{J}(g,s=9;8)$ along the unipotent radical of the Heisenberg parabolic yields Kim's weight $8$ singular modular form on $GE_{7,3}$, which has rational Fourier coefficients \cite{kim}.  By applying Proposition \ref{prop:CHrk1rk2}, one obtains that $E_{J}(g,s=9;8)$ has rational rank one and rank two Fourier coefficients as well.  Thus, $E_{J}(g,s=9;8)$ has rational rank zero, rank one, and rank two Fourier coefficients.
\end{enumerate}
We now split up the various pieces into subsections below.  Throughout this section, set $\zeta_\Theta(s) = \zeta(s)\zeta(s-3)$ \cite{ganATM}.

\subsection{The holomorphic Eisenstein series} Set $V' = H^2 \oplus \Theta$, a quadratic space of signature $(2,10)$ that comes equipped with an integral lattice $V'_0 = H_0^2 \oplus \Theta_0$.  In this subsection we construct and analyze a certain holomorphic spherical Eisenstein series on $\SO(V')$.  Let the bases of the two copies of $H$ be $e_1, f_1$ and $e_2, f_2$.

To define this Eisenstein series, we proceed as follows.  First, denote by $v_1, v_2$ an orthonormal basis of $V_2 = V'(\R)^{+}$.  For an even positive integer $\ell$, define the Schwartz function $\Phi_{\infty,\ell}$ on $V'(\R)$ as $\Phi_{\infty,\ell}(v) = (v_1 + iv_2, v)^{\ell} e^{-\pi ||v||^2}$.  Let $\Phi_f$ be the characteristic function of $V'_0(\widehat{Z})$ and set $\Phi = \Phi_f \otimes \Phi_{\infty,\ell}$, a Scwhartz-Bruhat function on $V'(\A)$.

We now set
\[f_{V',\ell}(g,\Phi,s) = \int_{\GL_1(\A)}{|t|^{s}\Phi(t f_2 g)\,dt}.\]
Denote by $P_{V'}$ the parabolic subgroup of $\SO(V')$ that stabilizes the line spanned by $f_2$.  The Eisenstein series $E_{V',\ell}(g,s)$ is defined as
\[E_{V',\ell}(g,s) = \sum_{\gamma \in P'(\Q)\backslash \SO(V')(\Q)}{f_{V',\ell}(\gamma g,\Phi,s)}.\]
The sum converges absolutely when $Re(s) > 10$.  The purpose of this subsection is to prove the following proposition.
\begin{proposition}\label{prop:holEis} The Eisenstein series $E_{V',8}(g,s)$ is regular at $s=8$, and is the automorphic function associated to a holomorphic weight $8$ modular form on $\SO(V')$ with rational Fourier coefficients.\end{proposition}
As mentioned above, Proposition \ref{prop:holEis} is likely well-known; as we do not know of a precise reference, we give a brief sketch of the proof.

\begin{proof} Denote by $P_{0,V'}$ the minimal parabolic of $\SO(V')$ that stabilizes the flag
\[ V' \supseteq \mathrm{Span}(e_2,\Theta,f_2,f_1) \supseteq \mathrm{Span}(\Theta,f_2,f_1) \supseteq \mathrm{Span}(f_2,f_1) \supseteq \mathrm{Span}(f_1) \supseteq 0.\]

We begin by computing the constant term of $E_{V',8}(g,s)$ to $P_{0,V'}$.  Denote by $r_1, r_2$ characters of the diagonal split torus of $\SO(V')$, so that the positive roots associated to $P_{0,V'}$ are $\{r_1-r_2, r_1, r_2, r_1+r_2\}$.  Denote the simple reflections associated to these positive roots by $w_{12}$ and $w_2$, in obvious notation.  The constant term of $E_{V',8}(g,s)$ along $P_{0,V'}$ is the sum of four terms of the form $M(w,s)f_{V',8}(g,\Phi,s)$, for $w = 1, w_{12}, w_2 w_{12}, w_{12} w_2 w_{12}$.

Define $\lambda_s = (s-5)r_1 + (-4)r_2$.  Then the inducing section $f_{V',\ell}(g,\Phi,s)$ defines an element of $Ind_{P_{0,V'}}^{\SO(V')}(\delta_{P_{0,V'}}^{1/2} \lambda_s)$ that is spherical at every finite place, but not spherical at infinity.  Applying the Weyl group elements $w$, the character $\lambda_s$ is moved as follows:
\begin{itemize}
\item $\lambda_s = (s-5)r_1 + (-4)r_2$
\item $\stackrel{w_{12}}{\mapsto} (-4)r_1 + (s-5)r_2$, $\frac{\zeta(s-1)}{\zeta(s)} \frac{\Gamma_{\R}(s-1)}{\Gamma_\R(s)} \frac{\left(\frac{2-s}{2}\right)_{\ell/2}}{\left(\frac{s}{2}\right)_{\ell/2}}$
\item $\stackrel{w_{2}}{\mapsto} (-4)r_1 + (5-s)r_2$, $\frac{\zeta_\Theta(s-5)}{\zeta_\Theta(s-1)} \frac{\Gamma(s-5)}{\Gamma(s-1)} = \frac{\zeta(s-5)\zeta(s-8)}{\zeta(s-1)\zeta(s-4)} \frac{\Gamma(s-5)}{\Gamma(s-1)}$
\item $\stackrel{w_{12}}{\mapsto} (5-s)r_1 + (-4)r_2$, $\frac{\zeta(s-9)}{\zeta(s-8)} \frac{\Gamma_{\R}(s-9)}{\Gamma_\R(s-8)} \frac{\left(\frac{10-s}{2}\right)_{\ell/2}}{\left(\frac{s-8}{2}\right)_{\ell/2}}$
\item $= \lambda_{10-s}$.
\end{itemize} 
We have also included above the $c(w,s)$-factors introduced by the intertwining operators.  Plugging in $\ell = 8$, one finds that the above $c(w,s)$-functions are $0$ at $s=8$, using that $\zeta(0)$ is finite and nonzero.

It follows that the constant term of $E_{V',8}(g,s=8)$ to $P_{V',0}$ consists only of $f_{V',8}(g,\Phi,s=8)$.  Moreover, for $g = g_\infty \in \SO(V')(\R)$, one has
\[f_{V',8}(g,\Phi,s=8) = \pi^{-8} \zeta(8) \Gamma(8) \frac{(f_2 g, v_1 + iv_2)^8}{||f_2 g||^{16}} =  \frac{C_8}{(f_2 g, v_1-iv_2)^{8}}\]
for a nonzero rational number $C_8$.  

From the above facts one can deduce that $E_{V',8}(g,\Phi,s=8)$ corresponds to a holomorphic modular form on $\SO(V')$ of weight $8$ with rational Fourier coefficients.
\end{proof}

\subsection{The Eisenstein series on $\SO(H^3 \oplus \Theta)$} Set $V = H^3 \oplus \Theta = H \oplus V'$, with $H$ spanned by $e_1, f_1$ and $V'= \mathrm{Span}\{e_2,e_3, \Theta, f_3, f_2\}$.  The fixed integral lattice in $V$ is $V_0 = H_0^3 \oplus \Theta_0$. Denote by $\Phi_f$ the characteristic function of the lattice $V_0(\widehat{\Z})$ in $V(\A_f)$. In this subsection, we analyze the Eisenstein series $E_{8}(g,\Phi_f,s)$ on $\SO(V)$ at $s=9$, which is outside the range of absolute convergence.  In particular, we prove the following proposition.

\begin{proposition}\label{prop:B3type8} The Eisenstein series $E_{8}(g,\Phi_f,s)$ is regular at $s=9$ and defines a modular form of weight $8$ on $\SO(V)$ at this point.  Moreover, $\pi^{-8}E_8(g,\Phi_f,s=9)$ has rational Fourier expansion.\end{proposition}
\begin{proof}
 We consider the constant term of $E_{8}(g,\Phi_f,s)$ to the Levi subgroup $\GL_1 \times \SO(V')$.  There are three terms: The inducing section (supported on $x^8 y^8$), the Eisenstein series $E_{V',8}(g,s-1;8)$ (supported on $x^{16} + y^{16}$; see \cite[Proposition 3.3.2]{pollackE8}), and an intertwining operator $M(w_0,s)$ applied to the inducing section.  

As the inducing section is spherical at every finite place, the finite part of the intertwining operator $M(w_0,s)$ is computed easily.  For the finite places, one obtains for the function $c(w_0,s)$ 
\begin{align*} c(w_0,s) &= \frac{\zeta(s-1)}{\zeta(s)} \frac{\zeta(s-2)}{\zeta(s-1)} \frac{\zeta_\Theta(s-6)}{\zeta_\Theta(s-2)} \frac{\zeta(s-10)}{\zeta(s-9)} \frac{\zeta(s-11)}{\zeta(s-10)} \\ &= \frac{\zeta(s-6)\zeta(s-11)}{\zeta(s)\zeta(s-5)}.\end{align*}
Consequently, $c(w_0,s)$ vanishes at $s=9$.  Moreover, by Remark \ref{rmk:l=8}, the archimedean interwiner is finite and nonzero at $s=9$.  Therefore, the intertwined inducing section vanishes at $s=9$.  

Applying Proposition \ref{prop:holEis}, one obtains that $E_{8}(g,s=9;8)$ has constant term a sum of the inducing section $f(g,\Phi_f,s=9)$ and the holomorphic Eisenstein series $E_{V',8}(g,s=8;8)$.  It now follows from Corollary \ref{cor:eta0} that $D_8$ annihilates this constant term.  From Lemma \ref{lem:Dfl} one then concludes that $D_8 E_{8}(g,s=9;8) = 0$, proving that this Eisenstein series is a modular form of weight $8$ on $\SO(V)$.

To prove that the Fourier expansion of $E_{8}(g,s=9;8)$ is rational, we use Corollary \ref{cor:FTrank2}.  In particular, in Corollary \ref{cor:FTrank2}, one only needs the inequality $\ell \geq n/2$, not $\ell > n+1$.  In our case of interest, $\ell = 8$ and $n = 10$, so we may apply this result.  We claim that the rank one Fourier coefficients of $E_{8}(g,s=9;8)$ consists only of a single term, as in \eqref{eqn:etaInt}; the integral in \eqref{eqn:etaInt2} again vanishes.

To see that the integral \eqref{eqn:etaInt2} vanishes, one proceeds as follows.  First, because $\eta$ is isotropic, we assume without loss of generality that $\eta = r f_2$.  Because our inducing section is spherical, we can assume $r$ is an integer. Now by Corollary \ref{cor:FTrank2}--which we may apply as remarked above--the archimedean part of \eqref{eqn:etaInt2} vanishes at $s=9$.  Thus, we must show that the finite adelic part does not give rise to a pole at this value of $s$.  To see this, for this particular $\eta = r f_2$, one can compute the integral \eqref{eqn:etaInt2} directly, by factoring it as a spherical intertwining operator and then a one-dimensional character integral.  The spherical intertwining operator is $M(w,s) = M(w_{23} w_{e_3} w_{23} w_{12},s)$, in the notation of Proposition \ref{prop:longInter}.  This intertwining operator gives a function $c(w,s)$ as
\begin{align*} c(w,s) &= \frac{\zeta(s-1)}{\zeta(s)} \frac{\zeta(s-2)}{\zeta(s-1)} \frac{\zeta_\Theta(s-6)}{\zeta_\Theta(s-2)} \frac{\zeta(s-10)}{\zeta(s-9)} \\ &= \frac{\zeta(s-6)\zeta(s-9)}{\zeta(s)\zeta(s-5)}\end{align*}
which is finite and nonzero at $s=9$.  The one-dimensional character integral produces a factor of $\frac{\sigma_{s-11}(n)}{\zeta(s-10)}$, which again is finite and non-zero at $s=9$.  Altogether, one sees that the integral \eqref{eqn:etaInt2} vanishes at $s=9$, as desired.

Because of this vanishing, and again because Corollary \ref{cor:FTrank2} applies in the case $\ell = 8$, $n=10$, the calculation of the Fourier coefficients of $E_{8}(g,s=9)$ now proceeds exactly as in Theorem \ref{thm:EisAlgFC}, using Proposition \ref{prop:holEis} to treat the constant term.  Because the inducing section is spherical at every finite place, the Fourier coefficients are valued in $\Q \subseteq \overline{\Q}$.  This completes the proof of the proposition.
\end{proof}

\subsection{Intertwining operators and the modular form of weight $8$ on $G_J$} In this subsection we analyze the Eisenstein $E_J(g,s;8)$ that is spherical at every finite place. See \cite[section 2.2]{pollackE8} for this Eisenstein series. Let $P_0$ denote the minimal parabolic of $G_J$. The purpose of this subsection is to prove the following proposition.

\begin{proposition}\label{prop:JEis8} The Eisenstein series $E_J(g,s;8)$ is regular at $s=9$ and defines a square integrable modular form of weight $8$ at this point.  Moreover, its constant term along $P_0$ is a sum of two terms.\end{proposition}

Our proof of Proposition \ref{prop:JEis8} rests on the computation of several intertwining operators.  The rational root system of $G_J$ is of type $F_4$; let $\alpha_1, \alpha_2, \alpha_3, \alpha_4$ be the simple roots:
\[ \circ ---- \circ ==>== \circ ---- \circ;\]
the roots are labeled $1,2,3,4$ from left to right.  Let $\Phi_+$ denote the positive roots for this root system and $\Phi_{C_3}$ the roots inside the Levi of type $C_3$.  As is standard, let 
\[ [W_{F_4}/W_{C_3}] = \{w \in W_{F_4}: w(\Phi_{C_3} \cap \Phi_{+}) \subseteq \Phi_{+}\}\]
be the set of minimal length coset representatives.  Here $W_{F_4}$ is the Weyl group of the $F_4$ root system, and $W_{C_3}$ is the subgroup of $W_{F_4}$ generated by the simple reflections corresponding to the roots $\alpha_2, \alpha_3,\alpha_4$.  The set $[W_{F_4}/W_{C_3}]$ has $24$ elements.

As in \cite{ganATM}, we single out two special elements of $[W_{F_4}/W_{C_3}]$:
\begin{align*} w_0 &= [123214323412321] \\ w_{-1} &= [23214323412321], \end{align*}
of length $15$ and $14$ respectively.  Here the indices indicate how $w_0$, $w_1$ are expressed as a product of simple reflections. All other elements of $[W_{F_4}/W_{C_3}]$ have length less than $14$.  

Denote by $f_J(g,s;8)$ the inducing section used to define the Eisenstein series $E_J(g,s;8)$ that is spherical at every finite place.  For $w \in [W_{F_4}/W_{C_3}]$, we consider the intertwining operator
\[ M(w,s) f(g,s;8) = \int_{U_w(\A)}{f(w^{-1}ng,s;8)\,dn}.\]
Here $U_w$ is the unipotent group defined as
\[U_w = \prod_{\alpha > 0: w^{-1}(\alpha) < 0}{U_{\alpha}}\]
and $U_\alpha$ is the unipotent group associated to the rational root $\alpha$.  Note that if $\alpha$ is a long root, then $\dim U_\alpha = 1$, whereas if $\alpha$ is a short root then $\dim U_{\alpha} = 8$.

With notation as above, the content of this subsection is to prove the following proposition, which will be the main step in proving Proposition \ref{prop:JEis8}.

\begin{proposition}\label{prop:E8intert} Suppose $w \in [W_{F_4}/W_{C_3}]$.  Then
\begin{enumerate}
\item If $w \neq w_{0}, w \neq w_{-1}$, then $M(w,s) f(g,s;8)$ is finite at $s=20$.
\item If $w = w_{-1}$, then $M(w,s)f(g,s;8)$ has a simple pole at $s=20$.
\item If $w = w_{0}$, then $M(w,s)f(g,s;8)$ has a simple pole at $s=20$ and vanishes at $s=9$.
\end{enumerate}
\end{proposition}
\begin{proof} Let us first write down the long intertwiner $M(w_0,s)$.  At the finite places, one obtains \cite{ginzburgRallisSoudry}
\[c(w_0,s) = \frac{\zeta(2s-29) \zeta(s-28)\zeta(s-23) \zeta(s-19)}{\zeta(2s-28)\zeta(s)\zeta(s-5)\zeta(s-9)}.\]
The function $c(w_0,s)$ is finite nonzero at $s=20$ and $0$ at $s=9$.

At the archimedean place, one can compute $c_\infty(w_0,s)$ by factorizing $w_0 = [12321] \circ [43234] \circ [12321]$, and then using Proposition \ref{prop:longInter} to compute the $[12321]$ factors.  From now on, all archimedean intertwining operators are calculated up to exponential factors and nonzero constants.  The middle $[43234]$ intertwiner turns out to be spherical.  One obtains
\[c_\infty(w_0,s) = c_{8}^{B_3}(s-17) c_{mid}^{C_3}(s) c_{8}^{B_3}(s)\]
where
\[c_{mid}^{C_3}(s) = \frac{\Gamma(s-10)}{\Gamma(s-6)} \frac{\Gamma(s-14)}{\Gamma(s-10)} \frac{\Gamma_\R(2s-29)}{\Gamma_\R(2s-28)} \frac{\Gamma(s-15)}{\Gamma(s-11)} \frac{\Gamma(s-19)}{\Gamma(s-15)} = \frac{\Gamma(s-\frac{29}{2}) \Gamma(s-19)}{\Gamma(s-6)\Gamma(s-11)}\]
and $c_{\ell}^{B_3}(s)$ is from Proposition \ref{prop:longInter}.  In this case,
\[c_{8}^{B_3}(s) = \frac{\left(\frac{s-9}{2}\right)_{4}}{\left(\frac{s-2}{2}\right)_{5}}  \cdot \frac{\Gamma\left(s-6\right)}{\Gamma\left(s-2\right)} \cdot \frac{\left(\frac{s-18}{2}\right)_{4}}{\left(\frac{s-11}{2}\right)_{5}}.\]
Simplifying,
\[c_\infty(w_0,s) = \frac{\Gamma(s-6)\Gamma(s-23)}{\Gamma(s-2)\Gamma(s-19)} \cdot \frac{\left(\frac{s}{2}-9\right)_4 \left(\frac{s-17}{2} - 9\right)_4}{\left(\frac{s}{2}-1\right)_5 \left(\frac{s-11}{2}\right) \left(\frac{s-19}{2}\right)_5 \left(\frac{s}{2}-14\right)}.\]
This function is immediately checked to be finite and nonzero at $s=9$ and has a pole at $s=20$.  Combining with properties of $c(w_0,s)$, this gives part (3) of the proposition.

Most of the $w \in [W_{F_4}/W_{C_3}]$ give absolutely convergent adelic integrals at $s=20$.  There are $7$ that do not, and these $7$ have the following factorizations:
\begin{itemize}
\item $[4323412321]$
\item $[3214323412321]$
\item $w_{-1} = [23214323412321]$
\item $[214323412321]$
\item $[21323412321]$
\item $[14323412321]$
\item $w_0 = [123214323412321]$
\end{itemize}
We will explain in a bit of detail the computation of $M(w_{-1},s)$.  The computation of the other intertwining operators are completely analogous or simpler.

To record the computations, we use the standard Euclidean model of the $F_4$-root system.  Specifically, consider $\Z^4$, with inner product $x \cdot y = x_1 y_1 + x_2 y_2 + x_3 y_3 + x_4 y_4$, where $x = (x_1, x_2, x_3, x_4)_E$ and $y = (y_1,y_2,y_3,y_4)_E$.  We write the subscript `E' to indicate the implicit Euclidean inner product.  Now, set
\begin{itemize}
\item $\alpha_1 = (0,1,-1,0)_E$
\item $\alpha_2 = (0,0,1,-1)_E$
\item $\alpha_3 = (0,0,0,1)_E$
\item $\alpha_4 = \frac{1}{2}(1,-1,-1,-1)_E$.\end{itemize}
With the $\alpha_i$ as the simple roots, this gives a model of the $F_4$-root system.

Let $\lambda_s = (s-23,s-6,-5,-4)$, which we think of as an unramified character of $P_0$.  For this $\lambda_s$, one has $f(g,s;8) \in Ind_{P_0}^{G_J}(\lambda_s \delta_{P_0}^{1/2})$.  Although it is a bit more than is necessary to compute $M(w_{-1},s)$, we record how the long element $w_0$ moves around $\lambda_s$:
\begin{itemize}
\item $\lambda_s = (s-23,s-6,-5,-4)_E$
\item $[1]; (s-23,-5,s-6,-4)_{E}; s-1$
\item $[2]; (s-23,-5,-4,s-6)_E; s-2$
\item $[3]; (s-23,-5,-4,6-s)_E; s-6$
\item $[2]; (s-23,-5,6-s,-4)_E; s-10$
\item $[1]; (s-23,6-s,-5,-4)_E; s-11$
\item $[4]; (-13,-4,s-15,s-14)_E; s-10$
\item $[3]; (-13,-4,s-15,14-s)_E; s-14$
\item $[2]; (-13,-4,14-s,s-15)_E; 2s-19$
\item $[3]; (-13,-4,14-s,15-s); s-15$
\item $[4]; (6-s,s-23,-5,-4)_E; s-19$
\item $[1]; (6-s,-5,s-23,-4)_E; s-18$
\item $[2]; (6-s,-5,-4,s-23)_E; s-19$
\item $[3]; (6-s,-5,-4,23-s)_E; s-23$
\item $[2]; (6-s,-5,23-s,-4)_E; s-27$
\item $[1]; (6-s,23-s,-5,-4)_E; s-28$
\item $=\lambda_{29-s}$.\end{itemize}
In each line, the $[j]$ indicates that a simple reflection corresponding to the root $j$ has been performed, to get from the previous line to the current line.  The final parameter $s-k$ is the parameter needed to calculate a rational-rank-one intertwining operator, and is given as follows: If one has $[j] \mu' = \mu$, then the final parameter is the Euclidean inner product $\alpha_j \cdot \mu'$.  For example, in the line 
\[\bullet [2]; (s-23,-5,-4,s-6)_E; s-2,\]
one has $j=2, \mu = (s-23,-5,-4,s-6)_E$, $\mu' = (s-23,-5,s-6,-4)_{E}$, and 
\[\alpha_2 \cdot \mu' = (0,0,1,-1)_E \cdot (s-23,-5,s-6,-4)_E = (1)(s-6) + (-1)(-4) = s-2.\]

With the above data, and combining Proposition \ref{prop:SL3inter}, Proposition \ref{prop:longInter}, and the technique of \cite[page 26]{pollackE8}, one can compute the intertwining operators $M(w,s)$ without too much difficulty.  As mentioned, we will now detail the computation of $M(w_{-1},s)$.

First, the finite, spherical part of $M(w_{-1},s)$ is computed immediately from the terms $s-k$ of the above itemized data.  One gets
\[ c(w_{-1},s) = \frac{\zeta(2s-29)\zeta(s-27)\zeta(s-23)\zeta(s-19)}{\zeta(2s-28)\zeta(s)\zeta(s-5)\zeta(s-9)}.\]
The function $c(w_{-1},s)$ has a simple pole at $s=20$.

The archimedean calculation is, of course, more involved.  First, applying the first $10$ elements in the factorization of $w_{-1}$ gives an archimedean factor of $c_{mid}^{C_3}(s) c_8^{B_3}(s)$, in the notation above.  This product is
\[c_{mid}^{C_3}(s) c_8^{B_3}(s) = \frac{\Gamma(s-\frac{29}{2}) \Gamma(s-19)}{\Gamma(s-6)\Gamma(s-11)} \frac{\left(\frac{s-9}{2}\right)_{4}}{\left(\frac{s-2}{2}\right)_{5}}  \cdot \frac{\Gamma\left(s-6\right)}{\Gamma\left(s-2\right)} \cdot \frac{\left(\frac{s-18}{2}\right)_{4}}{\left(\frac{s-11}{2}\right)_{5}},\]
which is finite and nonzero at $s=20$.

Set $f_1 = x+y$, $f_2 = x-y$, as before Proposition \ref{prop:SL3inter}. The archimedean intertwiner $M_\infty([21],s)$ that comes after the $[4323412321]$ can now be computed by Proposition \ref{prop:SL3inter}.  One obtains a factor of $C_8(s-17)$, in the notation of that proposition, and the resulting inducing section is supported on $f_1^8 f_2^8$ at $g_\infty=1$.  The function $C_8(s-17)$ is\footnote{Although this function vanishes to first order $s=20$, the function $c([214323412321],s)$ has a simple pole at $s=20$, so there is no contradiction to the statement of the proposition.}
\[C_8(s-17) = \frac{\left(\frac{s-26}{2}\right)_4}{\left(\frac{s-17}{2}-1\right)_5}.\]
The next intertwining operator, an application of $M([3],s)$, is spherical, and gives a factor of $\frac{\Gamma(s-23)}{\Gamma(s-19)}$, which has a simple pole at $s=20$.  Thus, the archimedean intertwiner $M_\infty([3214323412321],s)$ is finite and nonzero at $s=20$.

To do the final application of an intertwining operator, the $M_\infty([2],s)$, we use the technique of \cite[page 26]{pollackE8}.  Expressing $f_1^8 f_2^8$ in the $x,y$ basis, one gets 
\[f_1^8 f_2^8 = C_4 (x^{16} + y^{16}) + C_3 (x^{14} y^2 + x^{2}y^{14}) + C_2 (x^{12} y^{4} + x^{4}y^{12}) + C_1 (x^{10} y^{6} + x^{6}y^{10}) + C_0 (x^{8} y^{8})\]
for nonzero constants $C_k$, $0 \leq k \leq 4$.  For $0 \leq k \leq 4$, on the term multiplying $C_k$ the operator $M_\infty([2],s)$ now produces factors of the form
\[\frac{\Gamma_\R(s-27)}{\Gamma_\R(s-26)} \frac{\left(\frac{28-s}{2}\right)_k}{\left(\frac{s-26}{2}\right)_k}.\]
For $0 \leq k \leq 3$ this is zero, while for $k = 4$ this is finite and nonzero.  This completes our analysis of $M(w_{-1},s)$, and with it, the proposition.
\end{proof}

We now complete the proof of Proposition \ref{prop:JEis8}.
\begin{proof}[Proof of Proposition \ref{prop:JEis8}] Using Proposition \ref{prop:E8intert}, the proof of Proposition \ref{prop:JEis8} proceeds exactly as the proof of Corollary 4.1.2 of \cite{pollackE8}.  The only thing left to remark upon is the square integrability of $E_J(g,s=9;8)$.  For this, we apply Jacquet's criterion \cite[I.4.11 Lemma]{moeglinWaldspurger}.  Writing in terms of the simple rational roots, one has 
\[\lambda_s = (s-23,s-6,-5,-4)_E = (2s-29)\alpha_1 + (3s-57)\alpha_2 + (4s-84)\alpha_3 + (2s-46)\alpha_4\]
and
\[[1]\lambda_s = (s-23,-5,s-6,-4)_E = (s-28)\alpha_1 + (3s-57)\alpha_2 + (4s-84)\alpha_3 + (2s-46)\alpha_4.\]
Plugging in $s=9$, one sees that all the exponents are negative in these characters, and thus $E_J(g,s=9;8)$ is square integrable.
\end{proof}

\subsection{Proof of Theorem \ref{thm:ntmRat}} The proof of this theorem, in its entirety, was outlined at the beginning of section \ref{sec:ntm}.  The only thing left to prove is step (4) of this outline and to discuss the constant term $E_J^{N_J}(g,s=9;8)$ of $E_J(g,s=9;8)$ along the unipotent radical of the Heisenberg parabolic.  

For step (4), note that the constant term $E_J^V(g,s;8)$ for $g$ in the Levi subgroup $L_J(\A)$ is a sum of Eisenstein series on $L_J$, one for each element of the double coset 
\[ P_J(\Q)\backslash G_J(\Q)/ Q_J(\Q) = W_{C_3} \backslash W_{F_4} \slash W_{B_3}.\]
The Eisenstein series associated to the double coset $P_J(\Q) 1 Q_J(\Q)$ is $E_8(g,s)$.  

At $s=9$, we have computed the constant term of each of these Eisenstein series down to $P_0$, and it is clear that they are identified, because the two terms contributing to the constant term of $E_J(g,s=9;8)$ along $P_0$ are those that come from the elements of length $0$ and $1$ of $[W_{C_3}\backslash W_{F_4}]$.  Consequently, at $s=9$, the difference $E_J^V(g,s;8) - E_8(g,s)$ has vanishing constant term along $P_0$.  Because the difference $E_J^V(g,s;8) - E_8(g,s)$ is a sum of Eisenstein series on $L_J$, we conclude $E_J^V(g,s=9;8) = E_8(g,s=9)$.  This proves step $(4)$ of the outline above.

The constant term $E_J^{N_J}(g,s=9;8)$ is analyzed in \cite[Corollary 3.5.1]{pollackE8}.  The holomorphic weight $8$ Siegel Eisenstein on $H_J = GE_{7,3}$ appears, along with the constant $\frac{\zeta(9)}{(2\pi)^{8}}$.  As mentioned previously, this weight $8$ Eisenstein series is analyzed in \cite{kim}, who proves that it has rational Fourier coefficients. This completes the proof of Theorem \ref{thm:ntmRat}.

\section{The minimal modular form}\label{sec:minD} In this final section, we discuss the minimal modular form on $G = \SO(3,8k+3)$ in certain special cases, and an application to the construction of a distinguished modular form on $G' = \SO(3,8k+2)$.  The results of this section are, in a sense, the analogues of the results in \cite{pollackE8} with quaternionic $E_8$ replaced by the classical group $D_{4k+3,3}$.

In more detail, the group $G$ supports a minimal modular form $\theta$, that is spherical at all finite places.  The purpose of this section is to recognize $\theta$ as the value of an Eisenstein series on $G$, to compute its Fourier expansion, and to show that the restriction $\theta' = \theta|_{G'}$ of $\theta$ to $G'$ is a distinguished modular form on $G'$.  The fact that $\theta'$ is distinguished is an example of the simplest ``lifting law" from \cite{pollackLL}, the one in section 2.2 of \emph{loc cit}, whereas the distinguished and singular modular forms constructed in \cite{pollackE8} use the (more complicated) lifting laws considered in section 7 and section 8 of \cite{pollackLL}.

Let $\Theta_0$ denote the ring of Coxeter's octonions, so that $\Theta_0$ is the even unimodular quadratic lattice of dimension $8$.  For an integer $k \geq 1$, let $V_0 = H_0^3 \oplus \Theta_0^k$ be the fixed lattice inside $V = H^3 \oplus \Theta^k$.  Set $V' = H^2 \oplus \Theta^k$ and set $G = \SO(V)$.  Fix a vector $\omega \in V'$ with $q'(\omega) < 0$.  Denote by $V_\omega = (\Q\omega)^{\perp}$ the orthogonal complement of $\Q\omega$, so that $V = \Q \omega \oplus V_{\omega}$.  Set $G' = \SO(V_\omega) \simeq \SO(3,8k+2)$ and define $V_\omega' = V_\omega \cap V'$.

Let $\Phi_f$ be the characteristic function of $V_0 \otimes \widehat{\Z} \subseteq V \otimes \A_f$ and $E(g,s) = \pi^{-4k} E_{4k}(g,\Phi_f,s)$ the Eisenstein series which was studied in detail in section \ref{sec:Eis}.  The purpose of this section is to prove the following theorem.

\begin{theorem}\label{thm:minTypeD} The Eisenstein series $E(g,s)$ is regular at $s= 4k+1$ and $\theta(g) := E(g,s=4k+1)$ is a modular form of weight $\ell = 4k$ at this point.  The modular form $\theta$ has rational Fourier expansion with all rank two Fourier coefficients equal to $0$.  The restriction $\theta' = \theta|_{G'}$ is a modular form on $G'$ of weight $\ell$.  It is distinguished in the sense that if $\eta \in V'_\omega$ has $q'(\eta) \neq 0$, then the Fourier coefficient $a_{\theta'}(\eta) \neq 0$ implies $q'(\eta) \in (\Q^\times)^2(-q(\omega)).$\end{theorem}
\begin{proof} Most of the work is to check the relevant intertwining operators, which show that $E(g,s)$ is regular at $s=\ell+1$ and defines a modular form of weight $\ell$ at this point.  To do this, we proceed as in section \ref{sec:ntm} and consider first the case of the Eisenstein series $E_{V',\ell}(g,s)$ on $\SO(V')$ evaluated at $s=\ell = 4k$.  

The constant term of the Eisenstein series $E_{V',\ell}(g,s)$ along the minimal rational parabolic consists of four terms: the inducing section, and intertwining operators $M(w,s)$ applied to the inducing section, where $w = w_{12}, w_{2}w_{12}, w_{12}w_2 w_{12}$ in the notation of Proposition \ref{prop:holEis}.  For ease of notation, set $m_0 = 4k$ and $\zeta_{an}(s) = \zeta(s)\zeta(s-m_0+1)$.  These intertwining operators produce functions $c(w,s)$ as follows:
\begin{enumerate}
\item $c(w_{12},s) = \frac{\zeta(s-1)}{\zeta(s)} \frac{\Gamma_\R(s)}{\Gamma_\R(s)} \frac{\left(\frac{2-s}{2}\right)_{\ell/2}}{\left(\frac{s}{2}\right)_{\ell/2}}$
\item $c(w_2 w_{12},s) = c(w_{12},s) \frac{\zeta_{an}(s-m_0-1)}{\zeta_{an}(s-1)}\frac{\Gamma(s-m_0-1)}{\Gamma(s-1)}$
\item $c(w_{12}w_2 w_{12},s) = c(w_{2}w_{12},s) \frac{\zeta(s-2m_0-1)}{\zeta(s-2m_0)}\frac{\Gamma_\R(s-2m_0-1)}{\Gamma_\R(s-2m_0)} \frac{\left(\frac{2m_0+2-s}{2}\right)_{\ell/2}}{\left(\frac{s-2m_0}{2}\right)_{\ell/2}}$.
\end{enumerate}
One checks easily that these functions $c(w_{12},s)$, $c(w_{2}w_{12},s)$ and $c(w_{12}w_2w_{12},s)$ vanish at $s=\ell=m_0$.  Consequently, as in Proposition \ref{prop:holEis}, $E_{V',\ell}(g,s)$ is regular at $s=\ell$ and is the automorphic function associated to a holomorphic modular form of weight $\ell$ on $\SO(V')$ with rational Fourier coefficients.

Next, to see that $E_{\ell}(g,\Phi_f,s)$ is regular at $s=\ell+1 = 4k+1$, we consider its constant term along the parabolic $P=MN$, just as in Proposition \ref{prop:B3type8}.  Just as in the proof of this proposition, to see that $E_{\ell}(g,\Phi_f,s)$ is regular at $s=\ell+1$ and defines a modular form of weight $\ell$ at this point, it suffices to check that the $c$-function associated to the long intertwiner $M(w_0,s)$ vanishes at $s=\ell+1$.  The finite part of this intertwiner gives
\begin{align*} c_f(w_0,s) &= \frac{\zeta(s-1)\zeta(s-2)\zeta_{an}(s-2-m_0)\zeta(s-2-2m_0)\zeta(s-3-2m_0)}{\zeta(s)\zeta(s-1)\zeta_{an}(s-2)\zeta(s-1-2m_0)\zeta(s-2-2m_0)} \\ &= \frac{\zeta(s-2-m_0)\zeta(s-3-2m_0)}{\zeta(s)\zeta(s-1-m_0)}.\end{align*}
This function vanishes at $s=\ell+1=m_0+1=4k+1$.  The archimedean part of this intertwining operator was computed in Proposition \ref{prop:longInter}.  One obtains
\[c_\infty(w_0,s) = \frac{\left(\frac{s-\ell-1}{2}\right)_{\ell/2}}{\left(\frac{s-2}{2}\right)_{\ell/2+1}} \frac{\Gamma(s-2-m_0)}{\Gamma(s-2)} \frac{\left(\frac{s}{2}-1-m_0-\ell/2\right)_{\ell/2}}{\left(\frac{s-3}{2}-m_0\right)_{\ell/2+1}}.\]
This function is finite and nonzero at $s=m_0+1=\ell+1$.  Consequently, $c(w_0,s) = c_f(w_0,s) c_\infty(w_0,s)$ vanishes at $s=\ell+1$, so that $E_{\ell}(g,\Phi_f,s)$ is regular at $s=\ell+1$ and defines a modular form of weight $\ell$ at this point.

The function $\theta(g)$ is defined to be the value $\pi^{-\ell}E_{\ell}(g,\Phi_f,s=\ell+1) = \pi^{-4k}E_{4k}(g,\Phi_f,s=4k+1)$.  By \cite[Theorem 1.1]{magaardSavin} and \cite[Proposition 4.1, Corollary 4.2]{savin}, the minimal representation of split $p$-adic $D_{4k+3}$ occurs as the spherical sub in $Ind_{P}^{G}(|\nu|^{4k+1}) = Ind_{P}^{G}(\delta_P^{1/2} |\nu|^{-1})$.  By these cited results, it follows that the modular form $\theta(g)$ has vanishing rank two Fourier coefficients.  

The rationality of the rank one Fourier coefficients of $\theta(g)$ is treated similarly to the rationality of the rank one Fourier coefficients of the Eisenstein series considered in \ref{prop:B3type8}, but with a little more work.  Specifically, the rationality follows from the vanishing of the integral \eqref{eqn:etaInt2}.  As in the proof of Proposition \ref{prop:B3type8}, to see that this integral vanishes, we factorize it into an intertwining operator and a one-dimensional character integral.  The intertwining operator is associated to the element $w'=w_{23} w_{3} w_{23} w_{12}$ of the Weyl group, in the notation of Proposition \ref{prop:longInter}.  The finite part $M_f(w',s)$ produces a $c$-function
\begin{align*} c_f(w',s) &= \frac{\zeta(s-1)\zeta(s-2)\zeta_{an}(s-2-m_0)\zeta(s-2-2m_0)}{\zeta(s)\zeta(s-1)\zeta_{an}(s-2)\zeta(s-1-2m_0)} \\ &= \frac{\zeta(s-2-m_0)\zeta(s-2-2m_0)}{\zeta(s)\zeta(s-1-m_0)},\end{align*}
which is finite and nonzero at $s=m_0+1$.  Similar to the evaluation of $M(w_{-1},s)$ in the proof of Proposition \ref{prop:E8intert}, the archimedean part of the intertwinter $M(w',s)$ produces $c$-functions
\[c_j(w',s)=\frac{\left(\frac{s-\ell-1}{2}\right)_{\ell/2}}{\left(\frac{s-2}{2}\right)_{\ell/2+1}} \frac{\Gamma(s-2-m_0)}{\Gamma(s-2)} \frac{\Gamma_{\R}(s-2-2m_0)}{\Gamma_\R(s-1-2m_0)} \frac{\left(\frac{3+2m_0-s}{2}\right)_{j}}{\left(\frac{s-1-2m_0}{2}\right)_{j}}\]
for integers $j$ with $0 \leq j \leq \ell/2=m_0/2$.  All of these functions $c_j(w',s)$ vanish at $s=m_0+1=\ell+1$.  Finally, the one-dimensional character integrals produce holomorphic functions of $s$ divided by $\zeta(s-2-2m_0) \Gamma_\R(s-2-2m_0)$.  This latter function is finite at $s=m_0+1$.  Combining this with the calculation of the functions $c(w',s)$, one sees that the integral \eqref{eqn:etaInt2} vanishes at $s=m_0+1$, as desired.  The rationality of the Fourier expansion of $\theta(g)$ follows.

The fact that $\theta'$ is a modular form of weight $\ell$ follows by a simple analysis of the differential operators $D_\ell$ on $G$ and on $G'$ as in Proposition \ref{prop:SO4const}.  Finally, that $\theta'$ is distinguished follows from the discussion in \cite[section 2.2]{pollackLL}.  This completes the proof of the theorem.
\end{proof}

\appendix

\section{Proofs of selected results}
This appendix contains the proofs of some of the results stated in the main body of the text but not proved there.

\subsection{Proofs from section \ref{sec:MFFE}}\label{subsec:appFE}

\begin{proof}[Proof of proposition \ref{thm:Dcoefs}] We first write out $\widetilde{D_{\ell}}$ in coordinates.  We obtain
\begin{align*}
\widetilde{D} F &= \sum_{j=1}^{n} D^M_{iv_1-v_2,u_j} F \otimes ((-1/2) y \otimes y \otimes u_j^{\vee}) + (u_+ \wedge u_j) F \otimes ((\sqrt{2}/4) (x \otimes y + y \otimes x) \otimes u_j^{\vee}) \\  & \,\, + D^{M}_{iv_1 +v_2,u_j}F \otimes ((-1/2) x \otimes x \otimes u_j^{\vee})  \\ & \,\, + ((iv_1 -v_2) \wedge u_{-})F \otimes ((-1/2) y \otimes y \otimes u_{-}^\vee) + (u_{+} \wedge u_{-} F) \otimes ((\sqrt{2}/4)( x \otimes y + y \otimes x) \otimes u_{-}^\vee) \\ &\,\, + ((iv_1 + v_2) \wedge u_{-} F) ((-1/2) x \otimes x \otimes u_{-}^{\vee}).\end{align*}

To compute the operator $D_{\ell}$ in coordinates, we must apply the contraction $\pr: \Vm_\ell \otimes \p^\vee \rightarrow (S^{2\ell}(Y_2) \oplus S^{2\ell-2}(Y_2)) \boxtimes V_{n+1}$.  With our coordinates $x,y$ this contraction is $[x^{\ell + v}][y^{\ell -v}] \otimes y \mapsto [x^{\ell+v-1}][y^{\ell-v}]$ and $[x^{\ell + v}][y^{\ell -v}] \otimes x \mapsto -[x^{\ell+v}][y^{\ell-v-1}]$.  For $D_{\ell}$ we therefore obtain

\begin{align*} 2 D_{\ell} F &= \sum_{j=1}^{n} u_j^\vee \otimes \left( \sum_{-\ell \leq v \leq \ell}{D^{M}_{iv_1 -v_2,u_j} F_v (-[x^{\ell+v-1}][y^{\ell-v}] \otimes y)}\right. \\ &\, + \sum_{-\ell \leq v \leq \ell}{\frac{\sqrt{2}}{2} (u_{+} \wedge u_{j}) F_v ([x^{\ell+v-1}][y^{\ell-v}] \otimes x - [x^{\ell+v}][y^{\ell-v-1}] \otimes y)} \\  &\, \left. + \sum_{-\ell \leq v \leq \ell}{D^{M}_{iv_1+v_2,u_j}F_v ([x^{\ell+v}][y^{\ell-v-1}] \otimes x)}\right) \\
&\, + u_{-}^\vee \otimes \left( \sum_{-\ell \leq v \leq \ell}{((iv_1-v_2) \wedge u_{-} F)_v (-[x^{\ell+v-1}][y^{\ell-v}] \otimes y)} \right. \\ &\,+ \sum_{-\ell \leq v \leq \ell}{\frac{\sqrt{2}}{2} (u_{+} \wedge u_{-}) F_v ([x^{\ell+v-1}][y^{\ell-v}] \otimes x - [x^{\ell+v}][y^{\ell-v-1}] \otimes y)} \\ &\, + \left. \sum_{-\ell \leq v \leq \ell}{((iv_1 + v_2) \wedge u_{-} F)_v ( [x^{\ell+v}][y^{\ell-v-1}] \otimes x)}.\right)
\end{align*}

From the fact that $E \cdot ([x^{\ell+v}][y^{\ell-v}]) = (\ell+v+1) [x^{\ell+v+1}][y^{\ell-v-1}]$ and $F \cdot ([x^{\ell+v}][y^{\ell-v}]) = (\ell-v+1) [x^{\ell+v-1}][y^{\ell-v+1}]$ we get
\[\left((iv_1 - v_2) \wedge u_{-}\right) F = \sum_{-\ell \leq v \leq \ell}{(-2 D^{V'}_{t m (iv_1-v_2)} F_v + \sqrt{2}(\ell+v)F_{v-1})[x^{\ell+v}][y^{\ell-v}]}\]
and
\[\left((iv_1 + v_2) \wedge u_{-}\right) F = \sum_{-\ell \leq v \leq \ell}{(-2 D^{V'}_{t m (iv_1+v_2)} F_v + \sqrt{2}(\ell-v)F_{v+1}) [x^{\ell+v}][y^{\ell-v}]}.\] 
The proposition follows. \end{proof} 

\begin{proof}[Conclusion of proof of Theorem \ref{thm:KBessel}] In subsection \ref{subsec:FEMF} we proved that the functions $\Wh_{\eta}$ of Definition \ref{def:Weta} satisfied the correct $(K \cap M)$-equivariance property.  We now complete the rest of the proof of Theorem \ref{thm:KBessel}.

We begin by considering the differential equations of Proposition \ref{thm:Dcoefs}.  Assume for now that $t > 0$.  From Proposition \ref{thm:Dcoefs}, we obtain that the $F_v$ satisfy the following differential-difference equations:
\begin{enumerate}
\item $(t\partial_t - (\ell+v))F_{v-1} = -u_\eta(t,m) F_v$
\item $(t\partial_t - (\ell-v+1))F_v = - u_\eta(t,m)^* F_{v-1}$
\item $D^{M}_{iv_1-v_2,u_j} F_v = -i\sqrt{2} t (\eta,m u_j) F_{v-1}$
\item $D^{M}_{iv_1+v_2,u_j} F_{v-1} = - i \sqrt{2} t (\eta, m u_j) F_v$.
\end{enumerate}

Define $F_v^0$ by the equality $F_v = t^{\ell+1} F_v^0$.  Then we obtain 
\begin{align*} ((t\partial_t)^2 - v^2)F_v^0 &= (t\partial_t - v)(t \partial_t +v) F_v^0 \\ &= (t\partial_t - v)\left(-u_\eta(t,m)^* F_{v-1}^0\right) \\ &= -u_\eta(t,m)^* (t\partial_t -v +1) F_{v-1}^0 \\ &= |u_\eta(t,m)|^2 F_v^0.\end{align*}
Because $F_v(t,m)$ is of moderate growth at $t \rightarrow \infty$, we deduce that $F_v^0(t,m) = C_v(m) K_v(|u_\eta(t,m)|)$ for some function $C_v(m)$ of $m$.

Now, because $(y\partial_y + v)K_v(y) = -yK_{v-1}(y)$, $(t\partial_t + v)K_v(t |\mu|) = -|\mu| t K_{v-1}(|\mu| t)$ if $\mu$ is independent of $t$.  Set $\mu = \sqrt{2} i(\eta,m (iv_1-v_2))$.  Thus
\begin{align*} -(|\mu| t) C_v(m) K_{v-1}(|\mu| t) &= (t\partial_t + v) F_v^0 \\ &=- u_\eta(t,m)^* F_{v-1}^0 \\ & = -u_\eta(t,m)^* C_{v-1}(m) K_{v-1}(|u_\eta(t,m)|).\end{align*}
Thus 
\[C_v(m) = \left(\frac{u_\eta(t,m)^*}{|u_\eta(t,m)|}\right) C_{v-1}(m) = \left(\frac{|u_\eta(t,m)|}{u_\eta(t,m)}\right) C_{v-1}(m).\]
We conclude that $C_v(m) = C_0(m) \left(\frac{|u_\eta(t,m)|}{u_\eta(t,m)}\right)^v$ for some function $C_0(m)$ that does not depend on $t$.

To see that $C_0(m) = C$ is a constant, independent of $m$, we use the final two differential equations involving $D^{M}_{iv_1\pm v_2, u_j}$.  One verifies immediately from the definitions that
\begin{equation}\label{eqn:Dtm0} D^{M}_{iv_1 \pm v_2,u_j}\left\{m \mapsto (\eta, m(iv_1-v_2))\right\} = -(iv_1 \pm v_2, iv_1-v_2) (\eta, m u_j).\end{equation}
This is $0$ for $D^M_{iv_1-v_2,u_j}$ and $2 (\eta, m u_j)$ for $D^{M}_{iv_1+v_2,u_j}$.  It follows that
\begin{equation}\label{eqn:Dtm1}D^M_{iv_1+v_2,u_j}(|u_\eta(t,m)|) = \frac{|u_\eta(t,m)|}{u_\eta(t,m)} \sqrt{2} ti(\eta, m u_j)\end{equation}
and
\begin{equation}\label{eqn:Dtm2}D^M_{iv_1-v_2,u_j}(|u_\eta(t,m)|) = \left(\frac{|u_\eta(t,m)|}{u_\eta(t,m)}\right)^{-1} \sqrt{2} ti(\eta, m u_j).\end{equation}
From \eqref{eqn:Dtm1} and \eqref{eqn:Dtm2} and the third and fourth enumerated equations applied to the case $v=0$, resp. $v=1$, one obtains that $D^M_{iv_1 \pm v_2, u_j} C_0(m) = 0$.  By the $(K \cap M)$-equivariance proved above, we know that $C_0(mk) = C_0(m)$ for all $k \in K \cap M$.  Combined with the differential equations $D^M_{iv_1 \pm v_2, u_j} C_0(m) = 0$ for all $j$, this gives that $C_0(m) = C$ is a constant, as desired.

Let us now check that $F(t,m) = \sum_{\ell \leq v \leq \ell}{F_v(t,m) \frac{x^{\ell+v} y^{\ell-v}}{(\ell+v)!(\ell-v)!}}$ with 
\[F_v(t,m) = t^{\ell+1} \left(\frac{|u_\eta(t,m)|}{u_\eta(t,m)}\right)^{v} K_v(|u_\eta(t,m)|)\]
satisfies the above differential equations on the connected component of the identity.  To see this, first note that the identity $(y\partial_y + v)K_v(y) = -yK_{v-1}(y)$ implies that the $F_{v}$ satisfy the second difference-differential equation.  Similarly, the identity $(y\partial_y - v)K_v(y) = -yK_{v+1}(y)$ implies that the $F_{v}$ satisfy the first difference-differential equation.  To see that this $F(t,m)$ satisfies the third and fourth difference-differential equations, first note that
\[D^M_{iv_1 +v_2,u_j}\left(\frac{|u_\eta(t,m)|}{u_\eta(t,m)}\right) = -\frac{|u_\eta(t,m)|}{u_\eta(t,m)^2} \sqrt{2} ti(\eta,mu_j).\]
Moreover, from $\partial_y K_v(y) = \frac{v}{y}K_v(y) - K_{v+1}(y)$, one obtains
\[D^M_{iv_1+v_2,u_j} K_v(|u_\eta(t,m)|) = \left(\frac{v}{|u_\eta(t,m)|} K_v(|u_\eta(t,m)|) - K_{v+1}(|u_\eta(t,m)|)\right) \left(\frac{|u_\eta(t,m)|}{u_\eta(t,m)}\right) \sqrt{2} ti(\eta,m u_j).\]
Combining these two equations gives
\[D^M_{iv_1+v_2,u_j}\left( \left(\frac{|u_\eta(t,m)|}{u_\eta(t,m)}\right)^{v} K_v(|u_\eta(t,m)|)\right) = -\sqrt{2} ti (\eta,m u_j) \left( \left(\frac{|u_\eta(t,m)|}{u_\eta(t,m)}\right)^{v+1} K_{v+1}(|u_\eta(t,m)|)\right)\]
which shows that the $F(t,m)$ satisfies the fourth enumerated differential equation.  The case of the third equation is similar.

Finally, we consider the condition $(\eta, \eta) \geq 0$.  By Lemma \ref{lem:etaPos}, we must check that $F(t,m) \equiv 0$ if there exists $m \in \SO(V')(\R)$ so that $(\eta, m(iv_1-v_2)) = 0$.  This follows by the argument of \cite[Proposition 8.2.4]{pollackQDS}.\end{proof}

\bibliographystyle{amsalpha}
\bibliography{QDS_Bib2} 

\newcommand{\etalchar}[1]{$^{#1}$}
\providecommand{\bysame}{\leavevmode\hbox to3em{\hrulefill}\thinspace}
\providecommand{\MR}{\relax\ifhmode\unskip\space\fi MR }
\providecommand{\MRhref}[2]{%
  \href{http://www.ams.org/mathscinet-getitem?mr=#1}{#2}
}
\providecommand{\href}[2]{#2}
\begin{thebibliography}{GGK{\etalchar{+}}19}

\bibitem[Gan00a]{ganATM}
Wee~Teck Gan, \emph{An automorphic theta module for quaternionic exceptional
  groups}, Canad. J. Math. \textbf{52} (2000), no.~4, 737--756. \MR{1767400}

\bibitem[Gan00b]{ganSW}
\bysame, \emph{A {S}iegel-{W}eil formula for exceptional groups}, J. Reine
  Angew. Math. \textbf{528} (2000), 149--181. \MR{1801660}

\bibitem[GGK{\etalchar{+}}19]{GGKPS}
Dmitry Gourevitch, Henrik P.~A. Gustafsson, Axel Kleinschmidt, Daniel Persson,
  and Siddhartha Sahi, \emph{Fourier coefficients of minimal and
  next-to-minimal automorphic representations of simply-laced groups}, 2019.

\bibitem[GGS02]{ganGrossSavin}
Wee~Teck Gan, Benedict Gross, and Gordan Savin, \emph{Fourier coefficients of
  modular forms on {$G_2$}}, Duke Math. J. \textbf{115} (2002), no.~1,
  105--169. \MR{1932327}

\bibitem[GR07]{GR}
I.~S. Gradshteyn and I.~M. Ryzhik, \emph{Table of integrals, series, and
  products}, seventh ed., Elsevier/Academic Press, Amsterdam, 2007, Translated
  from the Russian, Translation edited and with a preface by Alan Jeffrey and
  Daniel Zwillinger, With one CD-ROM (Windows, Macintosh and UNIX).
  \MR{2360010}

\bibitem[GRS97]{ginzburgRallisSoudry}
David Ginzburg, Stephen Rallis, and David Soudry, \emph{On the automorphic
  theta representation for simply laced groups}, Israel J. Math. \textbf{100}
  (1997), 61--116. \MR{1469105}

\bibitem[GW94]{grossWallach1}
Benedict~H. Gross and Nolan~R. Wallach, \emph{A distinguished family of unitary
  representations for the exceptional groups of real rank {$=4$}}, Lie theory
  and geometry, Progr. Math., vol. 123, Birkh\"auser Boston, Boston, MA, 1994,
  pp.~289--304. \MR{1327538}

\bibitem[GW96]{grossWallach2}
\bysame, \emph{On quaternionic discrete series representations, and their
  continuations}, J. Reine Angew. Math. \textbf{481} (1996), 73--123.
  \MR{1421947}

\bibitem[Hel84]{helgason}
Sigurdur Helgason, \emph{Groups and geometric analysis}, Pure and Applied
  Mathematics, vol. 113, Academic Press, Inc., Orlando, FL, 1984, Integral
  geometry, invariant differential operators, and spherical functions.
  \MR{754767}

\bibitem[Kar79]{karelAJM}
Martin~L. Karel, \emph{Functional equations of {W}hittaker functions on
  {$p$}-adic groups}, Amer. J. Math. \textbf{101} (1979), no.~6, 1303--1325.
  \MR{548883}

\bibitem[Kim93]{kim}
Henry~H. Kim, \emph{Exceptional modular form of weight {$4$} on an exceptional
  domain contained in {${\bf C}^{27}$}}, Rev. Mat. Iberoamericana \textbf{9}
  (1993), no.~1, 139--200. \MR{1216126}

\bibitem[KM11]{kobayashiMano}
Toshiyuki Kobayashi and Gen Mano, \emph{The {S}chr\"{o}dinger model for the
  minimal representation of the indefinite orthogonal group {${O}(p,q)$}}, Mem.
  Amer. Math. Soc. \textbf{213} (2011), no.~1000, vi+132. \MR{2858535}

\bibitem[KO03]{kOIII}
Toshiyuki Kobayashi and Bent {O}rsted, \emph{Analysis on the minimal
  representation of {${O}(p,q)$}. {III}. {U}ltrahyperbolic equations on
  {${R}^{p-1,q-1}$}}, Adv. Math. \textbf{180} (2003), no.~2, 551--595.
  \MR{2020552}

\bibitem[KS15]{kobayashiSavin}
Toshiyuki Kobayashi and Gordan Savin, \emph{Global uniqueness of small
  representations}, Math. Z. \textbf{281} (2015), no.~1-2, 215--239.
  \MR{3384868}

\bibitem[MgW95]{moeglinWaldspurger}
C.~M\oe~glin and J.-L. Waldspurger, \emph{Spectral decomposition and
  {E}isenstein series}, Cambridge Tracts in Mathematics, vol. 113, Cambridge
  University Press, Cambridge, 1995, Une paraphrase de l'\'{E}criture [A
  paraphrase of Scripture]. \MR{1361168}

\bibitem[MS97]{magaardSavin}
K.~Magaard and G.~Savin, \emph{Exceptional {$\Theta$}-correspondences. {I}},
  Compositio Math. \textbf{107} (1997), no.~1, 89--123. \MR{1457344}

\bibitem[Pol18a]{pollackLL}
Aaron Pollack, \emph{Lifting laws and arithmetic invariant theory}, Camb. J.
  Math \textbf{6} (2018), no.~4, 347--449.

\bibitem[Pol18b]{pollackE8}
\bysame, \emph{The minimal modular form on quaternionic ${E}_8$}.

\bibitem[Pol19a]{pollackQDS}
\bysame, \emph{The {F}ourier expansion of modular forms on quaternionic
  exceptional groups}, Duke Math. Journal (to appear) (2019).

\bibitem[Pol19b]{pollackG2}
Aaron Pollack, \emph{A quaternionic saito-kurokawa lift and cusp forms on
  ${G}_2$}, 2019.

\bibitem[Sav94]{savin}
Gordan Savin, \emph{Dual pair {$G_{J}\times{PGL}_2$} [where] {$G_{J}$} is the
  automorphism group of the {J}ordan algebra {${J}$}}, Invent. Math.
  \textbf{118} (1994), no.~1, 141--160. \MR{1288471}

\bibitem[Shi82]{shimura}
Goro Shimura, \emph{Confluent hypergeometric functions on tube domains}, Math.
  Ann. \textbf{260} (1982), no.~3, 269--302. \MR{669297}

\bibitem[Shu95]{shurman}
Jerry Shurman, \emph{Fourier coefficients of an orthogonal {E}isenstein
  series}, Pacific J. Math. \textbf{168} (1995), no.~2, 345--381. \MR{1339957}

\bibitem[SW11]{savinWeissman}
Gordan Savin and Martin~H. Weissman, \emph{Dichotomy for generic supercuspidal
  representations of {$G_2$}}, Compos. Math. \textbf{147} (2011), no.~3,
  735--783. \MR{2801399}

\bibitem[Wei06]{weissman}
Martin~H. Weissman, \emph{{$D_4$} modular forms}, Amer. J. Math. \textbf{128}
  (2006), no.~4, 849--898. \MR{2251588}

\end{thebibliography}

\end{document}